\documentclass[a4wide]{amsart}

\usepackage{lineno,hyperref}
\modulolinenumbers[5]

\usepackage[english]{babel}
\usepackage[latin1]{inputenc}
\usepackage[T1]{fontenc}
\usepackage{lmodern}
\usepackage{amsfonts}
\usepackage{stmaryrd}
\usepackage{amsmath,amsthm,amssymb}
\usepackage{array}
\usepackage{graphicx}
\usepackage{nicefrac}
\usepackage{enumerate}
\usepackage{color}

\usepackage{geometry}
\renewcommand{\(}{\left(}
\renewcommand{\)}{\right)}

\newcommand{\E}{\mathcal{E}}

\renewcommand{\L}{\mathcal{L}}

\newcommand{\R}{\mathcal{R}}

\newcommand{\Ord}{\mathcal{O}}
\renewcommand{\P}{\mathcal{P}}

\newcommand{\IN}{\mathbb{N}}

\newcommand{\IR}{\mathbb{R}}

\newcommand{\nti}{n \to \infty}
\newcommand{\bs}{\boldsymbol}
\newcommand{\True}{\mathtt{True}}
\newcommand{\False}{\mathtt{False}}

\newenvironment{ex}[0]{\vspace{3mm}\noindent \textbf{Example: }}{\vspace{3mm}}
\newenvironment{rmq}[0]{\vspace{3mm}\noindent \textbf{Remark: }}{\vspace{3mm}}

\newcommand{\mc}[1]{\ensuremath{\mathcal{#1}}}
\def\indi{\mbox{\hspace{0.2em}l\hspace{-0.55em}1}}

\newtheorem{thm}{Theorem}[section]
\newtheorem{prop}[thm]{Proposition}
\newtheorem{df}[thm]{Definition}
\newtheorem{lem}[thm]{Lemma}

\newcommand{\brem}{\begin{rmq}}
\newcommand{\erem}{\end{rmq}}

\title[tree size complexity and probability for Boolean functions]{The relation between tree size complexity and probability for Boolean functions
generated by uniform random trees$^\ast$}\thanks{$^\ast$) This research was partially supported by the P.H.C. AMADEUS project 29281NE, the \"OAD project
F03/2013 and the FWF grant SFB F50-03.}

\author{Veronika Daxner}

\author[Antoine Genitrini]{Antoine Genitrini$^\dag$}
\thanks{$^\dag$ Sorbonne Universit\'es, UPMC Univ Paris 06, CNRS, LIP6 UMR 7606, 4 place Jussieu 75005 Paris.
\url{Antoine.Genitrini@lip6.fr}}

\author[Bernhard Gittenberger]{Bernhard Gittenberger$^\ddag$}
\thanks{$^\ddag$ Technische Universit\"at
Wien, Wiedner Hauptstrasse 8-10/104, A-1040 Wien, Austria.
\url{gittenberger@dmg.tuwien.ac.at}}

\author[C\'ecile Mailler]{C\'ecile Mailler$^\P$}
\thanks{$^\P$ Department of Mathematical Sciences, University of Bath, BA2 7AY Bath, UK.
\url{c.mailler@bath.ac.uk}}

\begin{document}

\maketitle

\begin{abstract}
We consider a probability distribution on the set of Boolean functions in $n$ variables which is
induced by random Boolean expressions. Such an expression is a random rooted plane tree
where the internal vertices are labelled with connectives AND or OR 
and the leaves are labelled with variables or negated variables. We study the limiting
distribution when the tree size tends to infinity and derive a relation between the tree size
complexity and the probability of a function. This is done by first expressing trees representing
a particular function as expansions of minimal trees representing this function and then computing
the probabilities by means of combinatorial counting arguments relying on generating functions and
singularity analysis. 
\end{abstract}


\section{Introduction}

A Boolean function of $n$ variables is a mapping from $\{\True,\False\}^n$ onto $\{\True,\False\}$.
The interest in these objects dates back to the 1940ies when Riordan and Shannon
\cite{RS42,Sh49} discovered the so-called Shannon effect: the uniformly random $n$ variables Boolean function
has asymptotically almost surely exponential {\it complexity} when $n$ goes to infinity. 
Since then numerous papers have been devoted to developing a better
understanding of various aspects of Boolean functions. Concerning random Boolean functions and the
Shannon effect, further investigations were carried out by Lupanov~\cite{lupanov62, lupanov65} and
a proof based on simple combinatorial counting arguments is presented in Flajolet and Sedgewick's 
book~\cite{FlSe}. All these results concern the uniform probability distribution on the set of
Boolean functions in $n$ variables. 

In the last two decades people became interested in non-uniform distributions. 
A natural way to pick a Boolean function at random is to pick a Boolean formula 
at random and look at the function it represents.
It is convenient to see Boolean formulas as rooted plane trees whose 
internal nodes are labelled by logical connectives like $\land$, $\lor$ or $\implies$ 
(being respectively the conjunction, disjunction and implication operators) and
whose leaves are labelled by (possibly negated) variables $\{x_1, \ldots, x_n\}$ (where $n$ is a fixed integer).

The first efforts in generating a random Boolean function via a random Boolean formula/tree 
go back to Paris et al.~\cite{PVW94} and Lefmann and Savick{\'y}~\cite{LS97} 
where the authors pick a tree uniformly at random among all binary and/or 
(meaning that only the connectives $\land$ and $\lor$ are allowed) trees 
having $m$ leaves.
Lefmann and Savick\'y \cite{LS97} showed the existence of
a limiting distribution when the size $m$ of the tree (i.e. its number of leaves) 
tends to infinity and also bounded the probability of a Boolean function in terms of its complexity, 
the complexity of a Boolean function being the size of the smallest trees representing it. 
The existence of a limit distribution when $m$ goes to infinity
was shown independently for non binary trees by Woods~\cite{W97}. 
Woods' method is basically a special case of a more general
result, the Drmota-Lalley-Woods theorem originating in the works \cite{D97, L93, W97}. See
Flajolet and Sedgewick~\cite{FlSe} for an easy accessible formulation, 
Drmota~\cite{Dr09} for a detailed discussion, and
Fournier et al.~\cite{FGGG12} for an application in the context of Boolean formulas. 
The bounds of Lefmann and Savick\'y were later refined in Chauvin et al.\cite{CFGG04} 
by Analytic Combinatorics' methods. 
A survey of this topic was written by Gardy~\cite{G06}. 

Different but somehow related problems in the framework of balanced trees have been pursued by
Valiant~\cite{valiant} who wanted to generate a particular Boolean function with high probability.
His results were extended in Boppana~\cite{boppana85} and Gupta and Mahajan~\cite{gm97}. 
The existence of a limit distribution for
balanced and/or trees can be found in Fournier, Gardy and Genitrini~\cite{FGG09}.
The influence of different connectives was studied in Savick{\'y}~\cite{savicky90} and 
Brodsky and Pippenger~\cite{bp05}, but under different distributions. 

The study of the relation between the probability of a given Boolean function and its complexity under models
similar to the Lefmann-Savick{\'y}'s model was resumed recently by Kozik~\cite{kozik08} 
(and by Fournier et al.~\cite{FGGG08,FGGG12} for the implicational model) 
who was able to obtain asymptotic equivalents for the probability of a fixed Boolean function~$f$.
More precisely, they prove that if $p_{n,m}(f)$ is the probability that a uniform binary and/or tree of size~$m$ (and on~$n$ variables)
calculates $f$, then
\begin{equation}\label{eq:kozik}
\lim_{m\to+\infty} p_{n,m}(f) =: p_n(f) = \Theta\left(\frac1{n^{\ell(f)+1}}\right), \text{ when } n\to+\infty
\end{equation}
where $\ell(f)$ is the complexity of the Boolean function $f$.

The drawback of the models based on plane binary trees is that basic algebraic properties like
commutativity and associativity of the connectives are not incorporated into the model. The papers
\cite{GGKM15,GGKM12} extend the binary plane models to more general tree classes.
For instance, due to the associativity of the connectives $\land$ and $\lor$, 
it is very natural to consider non binary trees, but in that case there is no 
justification why the size of an and/or tree should be its number of leaves  
(as considered in~\cite{GGKM15})
and not its number of internal nodes or its total number of nodes.
And since the complexity of a Boolean function is defined as the size of the smallest trees calculating this function,
changing the notion of size changes the notion of complexity.

In computer science, the \emph{formula size complexity} (often called formula complexity or only
complexity) is an important quantity in the investigation of Boolean functions. When looking at
the tree representation of a Boolean formula, the formula size is the number of leaves of the tree.  
It is the size notion and thus the complexity notion used in the and/or trees literature mentioned 
above~\cite{PVW94, LS97, CFGG04, kozik08}.
This notion of size is natural in the binary connectives context (i.e. with gates of fanin~2)
because, in that context, the number of connectives is the number of leaves shifted by~1.
However, when one turns to Boolean circuits, and consequently to (computational) complexity theory questions
(see e.g. Graham, Gr\"{o}tschel and Lov\'{a}sz~\cite[Chapter~40]{GGL95}), the natural notion of size in this context
is the number of connectives (or gates).

In this article, we consider non binary and/or trees. The leaves are not shared like in circuits, 
and thus both their number and the number of connectives
are important when defining the size of a Boolean function.
Moreover, from a computer science point of view, the size of the storage is a function of the number of all vertices.
We thus define the {\it tree size} of a Boolean formula as the total number of nodes of its tree representation.
The complexity of Boolean functions associated to this new notion of size is called the {\it tree size complexity},
as opposed to the {\it formula size complexity} used in the literature.

In this article, we show that the typical uniform and/or tree of size $m$ changes drastically when changing the size notion.
This comes from the two successive limits taken in our set-up and in the literature: in Equation~\eqref{eq:kozik},
the size of the trees first goes to infinity and then, we let the number $n$ of variables labelling the trees go to infinity.
The typical and/or tree of size $m$ in the tree size model depends on the number of variables $n$:
the larger $n$ is, the more likely it is for the typical tree to have many leaves,
whereas in the formula size model, the shape of the typical tree of size $m$ does not depend on $n$.
This difference between the two models is the reason why we need to develop a whole new approach in this article since 
this difference of typical shapes make the proofs developed in the formula size literature collapse.

However, it eventually turns out that the two distributions induced on the set of Boolean functions 
by the tree size model and the formula size model very well fall under the same paradigms.
As in the formula size model (see Equation~\eqref{eq:kozik}), we can prove that there is a strong
relation between the (tree size) complexity and the probability of a function.
More precisely, if we pick uniformly at random an and/or tree of tree size $m$ (labelled on $n$ variables) and 
denote by $\mathbb P_{m,n}(f)$ the probability that this tree represents the Boolean function $f$, 
then we prove in this article that
\begin{equation}\label{eq:kozik_our}
\lim_{m\to+\infty} \mathbb P_{m,n}(f)=: \mathbb P_n(f) = \Theta\left(\frac1{n^{L(f)}}\right),
\end{equation}
where $L(f)$ is the tree size complexity.
This result is very close to Equation~\eqref{eq:kozik}, but the difference in the exponent of $n$
actually makes a clear difference, namely that in our model, the asymptotic probability of the two constant functions
$\True$ and $\False$ (which have complexity~$0$) is of constant order while it was tending to zero with $n$ in the formula size model.
This immediately implies that the tree size model does not exhibit the Shannon effect,
whereas proving that the formula size model does not exhibit the Shannon effect is proved but far from trivial 
(see~\cite{GG10,GGKM12,GGM14}).

\vspace{\baselineskip}
The present paper is organized as follows: In the next section we introduce the model and present
our main result. In Section~\ref{part:existence} we introduce the combinatorial setting
(generating functions for the basic tree classes) and briefly discusses the existence of the
limit distribution $\mathbb P_n$ (see Equation~\eqref{eq:kozik_our}). 
In Section~\ref{part:misc} is devoted to proving key properties of our model:
In particular, we investigate in this part the typical shape of the random uniform and/or tree of size $m$, 
when $m$ and then $n$ tend to infinity. This section shows how different the typical tree is our model
and the typical tree in the formula size model are different.
In Section~\ref{part:useful} we study a
particular subfamily of trees which will be used as an auxiliary structure in
Sections~\ref{part:tauto} and~\ref{part:general}. 
In Section~\ref{part:tauto}, we prove that the probability of the two constant 
functions is of constant order when $n$ tends to infinity, meaning that we prove Equation~\eqref{eq:kozik_our}
in the special cases $f=\True$ and $f=\False$. 
A fixed Boolean function divides the Boolean lattice into
larger and smaller functions (as well as non-comparable functions). 
This subdivision is quantified in Section~\ref{part:larger} and the
result will help us to estimate the limiting probability of literal functions, the simplest
non-constant functions, in Section~\ref{part:simple_x}. 
We finally have all ingredients needed to prove our main result (i.e. Equation~\eqref{eq:kozik_our}), 
which we do in Section~\ref{part:general}.

\section{Model and results}
\begin{df}
\label{df:assoctree}
An {\bf associative tree} is a rooted plane tree whose nodes have arity in $\mathbb{N}\setminus
\{1\}$, and such that each internal node is labelled by a connector AND (denoted by $\land$ in the
following) or by a connector OR (denoted by $\vee$ in the following) such that two identical
connectives cannot be neighbours (trees are {\bf stratified}), and where each leaf is labelled by
a literal taken from $\{x_1,\bar{x}_1,\ldots,x_n,\bar{x}_n\}$ (see Figure~\ref{fig:example} for an
example). 
\end{df}

\begin{figure}
\begin{center}
\fbox{\includegraphics[width=0.5\textwidth]{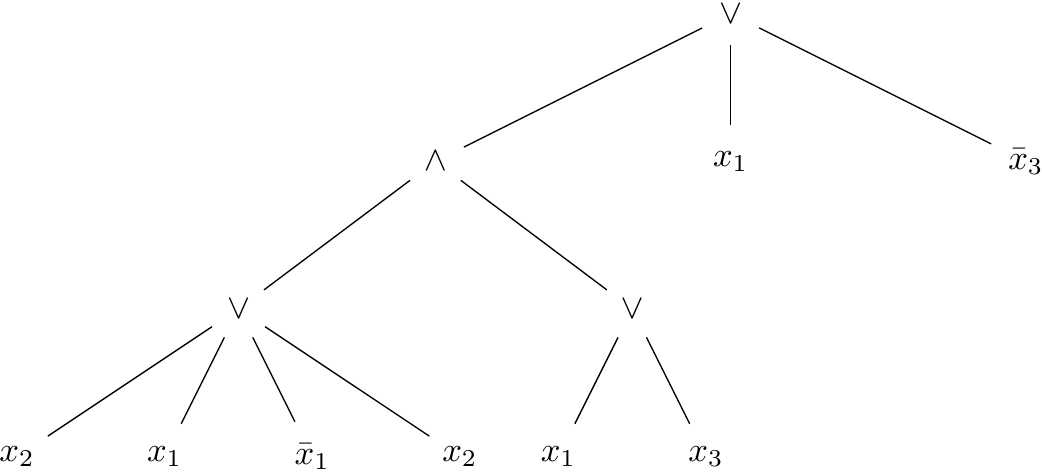}}
\end{center}
\caption{An associative tree which represents the constant function $\True$.}
\label{fig:example}
\end{figure}

By interpreting $\bar x_i$ as the negation of $x_i$, 
every such tree represents a Boolean formula (or expression) and therefore calculates a Boolean
function of $n$ variables. We denote by $\mathcal{F}_n$ the set of such Boolean functions.

\begin{df}
\label{df:complexity}
The {\bf size} $|t|$ of an associative tree $t$ is the number of its nodes (internal nodes and leaves).
We denote by $\mathcal{A}_{m,n}$ the set of associative trees of size $m$, and by $A_{m,n}$ the
cardinality of this set. Let $f\in \mathcal{F}_n$ be a Boolean function, its {\bf complexity} is
\begin{itemize}
\item 0 if $f$ is the function $\True : (x_1,\ldots,x_n)\mapsto \True$ or $\False : (x_1,\ldots,x_n)\mapsto \False$;
\item 2 if $f$ is a literal function, \emph{i.e.}\ there exist $x_i\in\{x_1,\ldots,x_n\}$ such that $f : (x_1,\ldots,x_n)\mapsto x_i$ or $f : (x_1,\ldots,x_n)\mapsto \bar{x}_i$;
\item the size of the smallest trees computing~$f$ if $f$ is neither a literal function nor a constant function. These trees of size $L(f)$ computing $f$ are called {\bf minimal trees} of $f$ and their set is denoted by $\mathcal{M}_f$.
\end{itemize}
\end{df}

\begin{ex}
The tree pictured in Figure~\ref{fig:example} has size 11.
The function $(x_1, \ldots, x_n)\mapsto x_1\text{ \sc xor }x_2 = (\bar{x}_1\land x_2) \lor (x_1\land \bar{x}_2)$ has complexity $7$.
\end{ex}

As already mentioned in the introduction, this notion of size differs from the notion used in the literature about and/or tree, 
where the size of a tree is the number of its leaves. Changing the notion of size also changes the notion of complexity of Boolean functions.
In all discussions, we will call the complexity defined in Definition~\ref{df:complexity} the {\it tree size} complexity, as opposed to the
{\it formula size complexity} of the literature.
As discussed in the introduction, this alternative definition is at least as natural and it is an interesting question to know how this change of definition will impact the distribution induced on the set of Boolean functions by the uniform distribution on the set of and/or trees of size $m$ and labelled on $n$ variables.

\begin{rmq}
The literal functions will be treated separately in the whole paper. But they could in fact be
treated as any other non-constant function by considering that its two minimal trees are the tree
rooted by an $\land$ with a single child labelled by $x_i$ and the tree rooted by an $\lor$ with a
single child labelled by $x_i$, even if those two trees are not associative trees according to
Definition~\ref{df:assoctree}!
\end{rmq}

\begin{df}
For all Boolean function $f$, we denote by $\mathbb{P}_{m,n}(f)$ the proportion of trees calculating $f$ among all trees of size $m$.
\end{df}

The aim of the present paper is to prove the following result, which sums up the asymptotic behaviour of this distribution $\mathbb{P}_{m,n}$ when the size $m$ of the considered tree tends to infinity.

\begin{thm}
\label{thm:general}
For all Boolean function $f$, the following limit exists and is positive 
\[\mathbb{P}_n(f) = \lim_{m\to\infty} \mathbb{P}_{m,n}(f).\]
Moreover, 
\begin{itemize}
\item There exist two constants $\alpha$ and $\beta$ such that, for all integer $n$,
\[0< \alpha \leq
\mathbb{P}_n(\True) = \mathbb{P}_n(\False) \leq\beta < \frac12.\]
\item If $f$ is a literal function, then
\[\mathbb{P}_n(f) = \Theta\left(\frac1{n^2}\right)\text{ when }n\to+\infty.\]
\item For all Boolean function $f$ such that $L(f)\geq 3$,
\[\mathbb{P}_n(f) = \Theta\left(\frac1{n^{L(f)}}\right)\text{ when }n\to+\infty.\]
\end{itemize}
\end{thm}

\begin{rmq}
Thanks to the \emph{ad hoc} definition of the complexity (\emph{cf.}
Definition~\ref{df:complexity}), the above theorem can be expressed as: 
For all Boolean function $f$,
\[\mathbb{P}_n(f) = \Theta\left(\frac1{n^{L(f)}}\right), \text{ as $n$ tends to infinity.}\]
\end{rmq}

As already mentioned in the introduction, our main result shows that changing the notion of size/complexity 
does not affect the behaviour of the induced distribution on Boolean functions: 
the main result above is very similar to Equation~\eqref{eq:kozik}, valid in the formula size complexity.
Therefore, one could expect that a slight generalisation of the proofs developed in the literature for the
formula size complexity might be enough to prove the above main result.
It might come as a surprise that generalising the formula size complexity proofs
collapse in our context due to the difference of the shapes of a typical uniform tree of size $m$ in both models 
- a glimpse of these differences will be given in Section~\ref{part:misc}, but roughly speaking, 
the number of leaves attached at the root is of order $\sqrt n$ while it did not depend on $n$ in the literature model.
Our main contribution is thus the development of a whole new approach to prove our main result.

\section{Existence of the limiting distribution $\mathbb{P}_n$}
\label{part:existence}

In this section, we prove the existence of the limiting distribution $\mathbb{P}_n$
on the set $\mathcal{F}_n$, which correspond to the first
assertion of Theorem~\ref{thm:general}. In the whole paper, we deeply use generating functions,
singularity analysis and the symbolic method. We refer the reader to~\cite{FlSe} for a
comprehensive introduction to this domain.

An associative tree can be formally described by a grammar: If $\hat{\mathcal{A}}$ (resp. $\check{\mathcal{A}}$) denotes the set of all associative trees rooted by an $\land$-connective (resp. an $\lor$-connective), $\mathcal{L}$ the set of literals and $\mathcal{C} = \{\land\}$, then $\hat{\mathcal{A}} =\mathcal{L} + \mathcal{C} \times \mbox{\tt seq}_{\geq 2}(\check{\mathcal{A}})$. Let us denote by $\hat{A}(z)$ (resp. $\check{A}(z)$) the generating function of associative trees rooted by an $\land$-connective (resp. an $\lor$-connective). The grammar can be translated into the functional equation
\[\hat{A}(z) = 2nz + z \cdot \frac{\check{A}(z)^2}{1-\check{A}(z)}.\]
Since, by symmetry, $\check{A}(z)= \hat{A}(z)$, we get 
\begin{equation}
\label{FG:A}
\hat{A}(z) = \frac{2nz+1 - \sqrt{(4n^2-8n)z^2-4nz+1}}{2(z+1)}.
\end{equation}
If we denote by $A(z)=\sum_{m\geq 0}A_{m,n}z^m$ the generating function of all associative trees,
then
\[A(z)= 2\hat{A}(z) - 2nz.\]

Moreover, let $\hat{A}_f(z)$ (resp. $\check{A}_f(z)$) denote the generating function of
associative trees rooted by an $\land$ (resp. an~$\lor$) connective and computing the Boolean
function $f$. Using the symbolic method, we get 
\begin{align*}
\hat{A}_f(z) 
&= z\cdot \indi_{f\text{lit}} + z\cdot \sum_{\ell \geq 2} \sum_{g_1 \land \cdots \land g_\ell = f} \check{A}_{g_1}(z) \ldots \check{A}_{g_\ell}(z)\\
\check{A}_f(z)
&= z\cdot \indi_{f\text{lit}} + z\cdot \sum_{\ell \geq 2} \sum_{g_1 \lor \cdots \lor g_\ell = f} \hat{A}_{g_1}(z) \ldots \hat{A}_{g_\ell}(z).
\end{align*}

We thus have a system of functionnal equations to which we will apply the Drmota-Lalley-Woods. Let
us check the hypothesis of this theorem: We refer the reader to~\cite[page~489]{FlOd} for a
precise statement of the Drmota-Lalley-Woods theorem for polynomial systems. A more general form
can be found in the original work \cite{D97} and in \cite[Section~2.2.5]{Dr09}. The above system
is non-linear, it has nonnegative coefficients and it satisfies a Lipschitz condition. It is
irreducible, because for every 
function $f$ we have $f\land \True \equiv f$ and $f\lor \True \equiv \True$. Finally, note that a
tree having a $\lor$-root and two subtrees, one being only one leaf labelled by $x_1$ and the
other one a tree of arbitrary size but with all leaves labelled by $\bar x_1$, is a
tautology. Thus, for all $m\geq 3$, there exists a tree of size $m$ calculating the constant
function $\True$ and hence $\check A_{\True}(z)$ is a-periodic. This is enough to imply that all the
functions $\hat A_f(z)$ and $\check A_f(z)$ are also a-periodic. Thus we have shown that all
hypotheses of the Drmota-Lalley-Woods are true.

We can therefore infer that all the generating functions $\hat A_f(z)$, $\check A_f(z)$ and $A(z)$ 
have the same unique singularity $\rho>0$ on their (common) circle of convergence, 
it is a square-root singularity and $\hat A_f(z)$, $\check A_f(z)$ as well as $A(z)$ admit a
singular expansion of the same type (Puiseux expansion in terms of powers of $\sqrt{\rho-z}$ 
at $\rho$. Applying a transfer lemma (see \cite{FlOd}) to $A(z)$, $\hat A_f(z)$, and
$\check A_f(z)$ we can conclude that the limiting distribution $\mathbb{P}_n(f) =
\lim_{m\to\infty} \frac{[z^m](\hat A_f(z)+\check A_f(z)}{[z^m]A(z)}$ exists for all $f$, 
proving the first statement of Theorem~\ref{thm:general}

\vspace{\baselineskip}
Knowing the generating function of $\land$--rooted associative trees (given in Equation~\eqref{FG:A}), 
we can state the following proposition that will be widely used later on.
\begin{prop}
\label{prop:sing}
The singularity $\rho$ of $\hat A(z)$ satisfies, as $\nti$,
\begin{equation}\label{rho_asym}
\rho = \frac12\cdot\frac1{n+\sqrt{2n}} =
\frac1{2n}-\frac1{n\sqrt{2n}}+\mathcal{O}\left(\frac1{n^2}\right),
\end{equation}
in particular, for all large enough $n$, we have $\frac1{2n}-\frac1{2n\sqrt n}\le \rho <\frac1{2n}$. 
Moreover,
\[A(\rho) = 1-\frac1{n} + \mathcal{O}\left(\frac1{n\sqrt{n}}\right) \quad \text{ and }\quad
\hat{A}(\rho) = 1-\frac1{\sqrt{2n}} + \mathcal{O}\left(\frac1{n}\right).\]
Finally, if $\hat{B}(z)$ denotes the generating function of $\land$--rooted associative trees, then 
$\hat{B}(z) = \hat{A}(z) - 2nz$. For all integer $n$,
\[\hat{B}(\rho)<\frac1{\sqrt{2n}} \quad\text{ and } \quad
\hat{B}(\rho) = \frac1{\sqrt{2n}} + \Ord \left(\frac{1}{n}\right)\text{ when }n\to+\infty.\]
\end{prop}

\section{Miscellaneous properties of the model.}
\label{part:misc}

In this section, we prove several propositions that lead to a better understanding of the model
and that will be useful throughout the paper. We prove among other results that the expected
number of leaves on the first level of a large associative tree behaves like $\sqrt{n}$, and that
there are few trees with no leaves on the first level. All these results are really surprising in
view of the formula size models of the literature where the number of leaves attached to the root of a typical tree of large size
does not depend on the number $n$ of variables.
Therefore, in the present model, it is quite likely to have a literal 
and its negation appearing as labels of two leaves attached to the root, 
making the whole tree calculate one of the two constant functions, 
whereas this configuration is very unlikely in the formula size model (see~\cite{GGKM15}).

\vspace{\baselineskip}
First of all, let us define the limiting ratio of a family:
\begin{df}
Let $\mathcal{T}$ be a family of associative trees and $T_m$ be the number of trees of size $m$ in this family. 
The limiting ratio of $\mathcal T$, if it exists, is defined (and denoted) by
\[\mu_n(\mathcal{T}) = \lim_{m\to\infty} \frac{T_m}{A_{m,n}}.\]
\end{df}

The following standard result will be extensively used in the following:
\begin{lem}
\label{lem:limiting_ratio}
Let $T(z)$ be the generating function of a family $\mathcal{T}$ of associative trees. 
Assume that $\rho$ (\emph{cf.}\ Proposition~\ref{prop:sing}) is the unique singularity of $T(z)$ on
its circle of convergence and that this singularity is of square-root type, 
\emph{i.e.}\ $T(z)$ admits a Puiseux expansion into powers of $\sqrt{\rho-z}$ at $\rho$. 
Then
\[\mu_n(\mathcal{T}) = \lim_{z\to \rho}\frac{T'(z)}{A'(z)}\]
where $z$ must move towards $\rho$ in such a way that $\arg (z-\rho)\neq 0$ and inside the domain
of analyticity of~$T(z)$ and~$A(z)$. 
\end{lem}

\begin{proof}
This is an immediate consequence of the fact that the asymptotic expansion of the derivative of an
analytic function is the derivative of the asymptotic expansion of the original function. 
\end{proof}

\begin{prop}
\label{prop:noleaf}
The limiting ratio of trees with no leaf on the first level is given by
\[\mu_n(\mathcal{A}^{(0)}) = \frac{1}{n\sqrt{2n}} + \mathcal{O}\left(\frac1{n^2}\right), \text{ as $\nti$.}\]
\end{prop}

\begin{proof}
The generating function of trees with no leaf on the first level is given by
\[A^{(0)}(z) = 2z\frac{\hat{B}(z)^2}{1-\hat{B}(z)}.\]
The limiting ratio of such trees is thus given by
\begin{align*}
\mu_n(\mathcal{A}^{(0)}) 
=& \lim_{m\to\infty}\frac{[z^m]A^{(0)}(z)}{[z^m]A(z)} = \lim_{z\to\rho} \frac{A'^{(0)}(z)}{A'(z)}\\
=& 2\rho \frac{2\hat{B}(\rho)}{1-\hat{B}(\rho)}\lim_{z\to\rho}\frac{\hat{B}'(z)}{A'(z)} 
+ 2\rho \frac{\hat{B}(\rho)^2}{(1-\hat{B}(\rho))^2}\lim_{z\to\rho}\frac{\hat{B}'(z)}{A'(z)}
+ 2 \frac{\hat{B}(\rho)^2}{1-\hat{B}(\rho)} \lim_{z\to\rho}\frac{1}{A'(z)}.
\end{align*}
Observe that the third term of the sum is equal to zero since $A'(z)$ tends to infinity when $z$
tends to $\rho$. This observation will be used in the whole paper and in the following, such terms
will be omitted without mentioning.
Note also that $\lim_{z\to\rho}\nicefrac{\hat{B}'(z)}{A'(z)} = \nicefrac12$. Finally, use the asymptotics given in Proposition~\ref{prop:sing} and get
\[\mu_n(\mathcal{A}^{(0)}) = \frac{1}{n\sqrt{2n}} + \mathcal{O}\left(\frac1{n^2}\right) \text{ as }\nti. \qedhere\]
\end{proof}

%\begin{rmq}
%Be careful that in the whole paper, the size $m$ of the trees tends to infinity, and then, the
%number of variables for the labelling tends to infinity. The order of these two limits must be kept
%in mind.
%\end{rmq}

\begin{prop}
\label{prop:A_alpha}
Let $\Gamma$ be a subset of $\gamma$ literals of $\{x_1,\bar{x}_1,\ldots,x_n,\bar{x}_n\}$. 
The limiting ratio of the set of all associative trees with at least one leaf on the first level
labelled by a literal from $\Gamma$ is given by
\[\mu_n(\mathcal{A}_{\Gamma}) = \gamma\sqrt{\frac{2}{n}} + \mathcal{O}\left(\frac1{n}\right), \text{when $n$ tends to infinity.}\]
\end{prop}

\begin{proof}
Let $A_{\Gamma}(z)$ be the generating function of associative trees with at least one leaf one the first level labelled by a literal from~$\Gamma$:
\[A_{\Gamma}(z) = \frac{2\gamma z^2}{(1-(\hat{A}(z)-\gamma z))(1-\hat{A}(z))} - 2\gamma z^2\]
because a tree in $\mathcal{A}_{\Gamma}$ has a root labelled by $\land$ or $\lor$ (which gives a 
factor $2z$), a first sequence of subtrees which are not a leaf labelled by a literal from
$\Gamma$ (which gives the factor $\frac{1}{1-(\hat{A}(z)-\gamma z)}$), then a leaf labelled by  literal from $\Gamma$ (factor $\gamma z$) and a sequence of arbitrary trees. Since sequences may be empty, this
construction also generates trees consisting of only two nodes. These have to be subtracted due to
the vertex degree constraints in associative trees. 
In view of Lemma~\ref{lem:limiting_ratio}, we know that
\[\mu_n(\mathcal{A}_{\Gamma}) = \lim_{m\to\infty}\frac{[z^m]A_{\Gamma}(z)}{[z^m]A(z)} =
\lim_{z\to\rho} \frac{A'_{\Gamma}(z)}{A'(z)} = \gamma\sqrt{\frac{2}{n}} +
\mathcal{O}\left(\frac1{n}\right). \qedhere\]
\end{proof}

\begin{prop}\label{prop:first_level_leaves}
Let $X_{m,n}$ be the number of leaves in the first level of an associative tree of size~$m$. 
Then $\lim_{m\to\infty} \mathbb E(X_{m,n})\sim 2\sqrt{2n}$, as $\nti$.
\end{prop}

\begin{proof}
Let us consider the bivariate generating function where $z$ marks the nodes and $u$ 
the leaves on the first level. We have the following equation:
\[A(z,u) = 2z\frac{(\hat{B}(z)+2nzu)^2}{1-(\hat{B}(z)+2nzu)}.\]
Therefore,
\[\frac{\partial}{\partial u}A(z,u)_|{u=1} = 8nz^2\frac{\hat{A}(z)}{1-\hat{A}(z)} + 4nz^2\frac{\hat{A}(z)^2}{(1-\hat{A}(z))^2}.\]
In view of Lemma~\ref{lem:limiting_ratio}, the expected number of nodes in the first level is given by 
\begin{align*}
\lim_{m\to\infty}\frac{[z^m]\frac{\partial}{\partial u}A(z,u)_|{u=1}}{[z^m]A(z)} 
&= \lim_{z\to\rho}\frac{\frac{d}{dz}\left(\frac{\partial}{\partial u}A(z,u)_|{u=1}\right)}{\frac{d}{dz}A(z)}\\ 
&= 2\sqrt{2n} + \mathcal{O}(1), 
\end{align*}
as $\nti$.
\end{proof}

\begin{rmq}
Using a similar calculation but plugging $u=e^{\nicefrac{it}{2\sqrt{2n}}}$ instead of
$u=1$, we can obtain even the limiting distribution. Indeed, one easily shows that the
distribution of the random variable $\nicefrac{X_n}{2\sqrt{2n}}$ converges to the
$\Gamma(2,\nicefrac12)$ distributed, as $\nti$, \emph{i.e.}, the distribution having density
$4xe^{-2x}\indi_{x\geq 0}$.
\end{rmq}
%A similar calculation yields 
%\begin{align*}
%\lim_{m\to\infty}\frac{[z^m]A(z,u)}{[z^m]A(z)} &=
%\lim_{z\to\rho}\frac{\frac{d}{dz}A(z,u)_{|u=1}}{\frac{d}{dz}A(z)}\\
%&\sim \frac{1}{(1-(u-1)\sqrt{2n})^2}. 
%\end{align*} 
%If we set $u=e^{\nicefrac{it}{2\sqrt{2n}}}$ and let $X_n$ denote the number of leaves in a random
%tree of size $n$, then we obtain
%\[
%\mathbb E e^{\nicefrac{it X_n}{2\sqrt{2 n}}} =
%\lim_{m\to\infty}\frac{[z^m]A\(z,e^{\nicefrac{it}{2\sqrt{2n}}}\)}{[z^m]A(z)} =
%\frac{1}{(1-\frac{it}{2})^2}
%\]
%which is the characteristic function of the $\Gamma(2,\nicefrac12)$ distribution.

\section{A useful family of trees}
\label{part:useful}

In the following, we will need some information about a specific class of trees which will serve
as an auxiliary construction for our further investigations. We start with the definition of this 
particular family of trees and then study its limiting ratio.

\begin{df}
Let $k, r, \ell\geq 0$ be three integers, let $\Gamma =\{\gamma_1,\ldots,\gamma_p\}$ be a subset of literals 
(with no occurrence of both a variable and its negation).
Let $\mathcal{M}_{k,\ell,r}^{\Gamma}$ be the family of $\lor$-rooted trees
\begin{itemize}
\item with exactly $k$ different literals $\alpha_1,\ldots,\alpha_k$ appearing as labels of leaves on the first level, 
such that both a variable and its negation cannot appear, and
for all $i=1,\dots,k$ and $j=1,\dots,p$ we have $\alpha_i\neq \gamma_j$,
\item with the root having exactly $\ell$ non-leaf subtrees,
\item and with at least one non-leaf subtree chosen from the family $\mathcal{J}_{k,r}^{\Gamma}$:
\end{itemize}
The family $\mathcal{J}_{k,r}^{\Gamma}$ contains all $\land$-rooted trees such that 
%there exists a set
%$\{\beta_1,\ldots,\beta_r\}$ of $r$ pairwise different labels (and different from the
%$\alpha_1,\ldots,\alpha_k,\gamma_1,\ldots,\gamma_p$ as well as their negations) such that 
\begin{itemize}
\item there are $r$ leaves in the first level carrying pairwise different labels
$\beta_1,\ldots,\beta_r$ (with no occurrence of both a variable and its negation) 
which are different from
$\alpha_1,\ldots,\alpha_k,\gamma_1,\ldots,\gamma_p$ as well as their negations; 
\item all the other leaves on the first level have labels from $\{\alpha_1,\ldots,\alpha_k,\gamma_1,\ldots,\gamma_p,\beta_1\ldots,\beta_r\}$ or their negations,
%\item the labels $\beta_1,\ldots,\beta_r$ all appear in the first level.
\end{itemize}
\end{df}

\begin{lem}
\label{lem:JM}
Let $J(z)$ and $M(z)$ denote the generating functions associated with $\mathcal{J}_{k,r}^{\bs
\gamma}$ and $\mathcal{M}_{k,\ell,r}^{\Gamma}$, respectively. Then 
\begin{equation} \label{gfJ}
J(z) = \frac{z^{r+1}}{(1-(\hat{B}(z) + 2(k+p)z))\ldots(1-(\hat{B}(z) + 2(k+p+r)z))}
\end{equation} 
and 
\begin{align}
[z^m]M(z)&\leq  \binom{n-p}{k+r}\binom{k+r}{k} \frac{2^{k+r} k!r!}{(\ell-1) !} \nonumber\\
&\quad\times  [z^{m-1}] \left(z^{k+\ell}\prod_{\nu =1}^k \frac1{1-\nu z}\right)^{(\ell)}
\frac{z^{r+1}}{\prod_{\nu =0}^{r}(1-(\hat{B}(z) + (k+p+\nu)z))} \hat{B}(z)^{\ell-1},
\label{eq:M}
\end{align}
for all $m\ge 0$, where $f^{(\ell)}$ denotes the $\ell$th derivative of $f$.
\end{lem}

\begin{proof}
The first equation is obvious from the definition of $\mathcal{J}_{k,r}^{\bs \gamma}$. 
Since Equation~\eqref{eq:M} is an inequality and not an equality 
we can do the following reasoning without worrying about double-counting some trees.
Let us follow the definition of $\mathcal{M}_{k,\ell,r}^{\Gamma}$:
First note that there are $\binom{n-p}{k+r} \binom{k+r}{k} 2^{k+r} r!k!$ ways to choose 
$\alpha_1,\ldots,\alpha_k$ and $\beta_1,\ldots,\beta_r$ (and their respective order) such 
that a variable and its negation cannot both be chosen. The leaves of the first level form a
sequence of $\alpha_1$ followed by the first occurrence of $\alpha_2$, then a sequence of leaves
with labels in $\{\alpha_1,\alpha_2\}$ followed by the first occurrence of $\alpha_3$, and so on. 
This corresponds to the generating functions $z^{k}\prod_{\nu=1}^k \frac1{1-\nu
z}$ and if we choose the places of the $\ell$ 
non-leaf subtrees of the root between the leaves of the first level, then we have to apply the
operator\footnote{Recall that 
$[z^m]\frac{(z^\ell f(z))^{(\ell)}}{\ell!} =
\binom{m+\ell}{\ell}f_m$.} $\frac1{\ell!} \frac{\partial^\ell}{\partial z^\ell}$. The generating
function of the non-leaf subtree taken from $\mathcal{J}_{k,\ell,r}^\Gamma$ and its place among
the non-leaf subtrees is $\ell J(z)$ and that of the other non-leaf subtrees is
$\hat{B}(z)^{\ell-1}$. When we collect all these terms and take into account that multiple
counting occurs in the cases where several non-leaf subtrees of the root are in
$\mathcal{J}_{k,\ell,r}^\Gamma$, then we get 
\begin{equation}\label{equation_only_inequality}
[z^m]M(z) \leq [z^m] z\binom{n-p}{k+r}\binom{k+r}{k} 2^{k+r} k! \frac1{\ell
!}\left(z^{k+\ell} \prod_{\nu=1}^k \frac1{1-\nu z}\right)^{(\ell)} \ell J(z) \hat{B}(z)^{\ell-1},
\end{equation} 
where the factor $z$ at the very beginning stands for the root. Finally, using Equation~\eqref{gfJ} to
replace $J(z)$ in Equation~\eqref{equation_only_inequality} gives Equation~\eqref{eq:M} as desired.
\end{proof}

By means of the previous lemma we can express the limiting ratio of the family
$\mathcal{M}_{k,\ell,r}^\Gamma$ in the form 
\begin{equation} \label{limratM}
\lim_{z\to\rho} \frac{M'(z)}{A'(z)} =
\binom{n-p}{k+r}\binom{k+r}{k} 2^{k+r} \frac{2^{k+r} k!r!}{(\ell-1) !}
\left(\rho^{k+\ell}\prod_{\nu=1}^k \frac1{1-\nu\rho}\right)^{(\ell)} \rho^{r+2} \lim_{z\to\rho} 
\frac{K'(z)}{A'(z)} 
\end{equation} 
where $K(z) = \frac{\hat{B}(z)^{\ell-1}}{\prod_{\nu=0}^{r}(1-(\hat{B}(z) + (k+p+\nu)z))}$. 
The only term depending on $z$ is $\nicefrac{K'(z)}{A'(z)}$. The next lemma analysis this term. 

\begin{lem}
\label{Klem}
Let $K(z)$ be as above. Moreover, assume that $k+p+r<n$ and $r=\mathcal{O}(\sqrt n)$. Then 
\[
\lim_{z\to\rho} \frac{K'(z)}{A'(z)} = \mathcal{O}\left(r2^r\ell \left(\frac1{2n}\right)^{\frac{\ell
-1}{2}}\right), 
\]
as $\nti$. 
\end{lem}

\begin{proof}
We have, as $z\to \rho$,  
\begin{equation} \label{Kauxil}
\frac{K'(z)}{A'(z)} \sim \frac{\hat{B}'(z)}{A'(z)}\frac1{\prod_{m=0}^r
(1-\hat{B}(\rho)-(k+p+m)\rho)} \left[(\ell -1)\hat{B}(\rho)^{\ell -2} + \sum_{m=0}^r
\frac{\hat{B}(\rho)^{\ell -1}}{1-\hat{B}(\rho)-(k+p+m)\rho}\right]. 
\end{equation} 
Moreover, we know that $\lim_{z\to\rho}\frac{\hat{B}'(z)}{A'(z)} =\nicefrac12$ and that 
$\hat{B}(\rho) \le \frac1{\sqrt{2n}}$ (\emph{cf.}\ Proposition~\ref{prop:sing}). 
Hence 
\[
\frac1{\prod_{m=0}^r
(1-\hat{B}(\rho)-(k+p+m)\rho)} \le \left(\frac12\left(1-\frac{\sqrt2}{\sqrt n}\right)\right)^{-r}
=\mathcal{O}(2^r).
\]
we get
\[
\frac{K'(z)}{A'(z)} = \mathcal{O}\left(r2^r\ell \left(\frac1{2n}\right)^{\frac{\ell -1}{2}}\right).
\] 
The term in brackets in Equation~\eqref{Kauxil} consists of $\mathcal{O}(r)$ terms, each bounded by $\ell
\hat{B}(\rho)^{\ell -1}$ which immediately yields the assertion. 
\end{proof}

\begin{lem}
\label{lem:Mklr}
If $k=\Omega(n^{\nicefrac14})$, $\ell = \mathcal{O}(n^{\nicefrac18})$ and $r \le \ell$, %\mathcal{O}(n^{\nicefrac18})$
then the limiting ratio of the family $\mathcal{M}_{k,\ell,r}^{\Gamma}$
satisfies $\mu_n(\mathcal{M}_{k,\ell,r}^{\Gamma}) =
\mathcal{O}\left(\frac1{n^{\nicefrac32}}\right)$, as $n$ tends to infinity.
\end{lem}

\begin{proof}
Turning back to \eqref{limratM} we see that the last factor is covered by Lemma~\ref{Klem}. In
order to estimate the penultimate factore, let $f(z) = z^{k+\ell}\prod_{\nu=1}^k \nicefrac1{(1-\nu z)}$ and consider a circle of radius $n^{-3/2}$
centered at $\rho$. Since $f(z)$ is a power series with non-negative coefficients, Cauchy's
estimate gives 
\[
\left|\frac{f^{(\ell)}(\rho)}{\ell!}\right|\le f\(\rho+\frac1{n^{\nicefrac32}}\) 
n^{\nicefrac{3\ell}2}. 
\]
Now, performing the substitution $z=\rho+\frac1{n^{3/2}}$ into $f(z)$ on the right-hand
side, using the monotonicity of $f$ and the inequality $\rho<\nicefrac1{2n}$ (\emph{cf.}
Proposition~\ref{prop:sing}), we obtain
\begin{align*} 
\left|\frac{f^{(\ell)}(\rho)}{\ell!}\right|&=\(\rho+\frac1{n^{\nicefrac32}}\)^{k+\ell}
n^{\nicefrac{3\ell}2} \prod_{\nu=1}^k\frac1{1-\nu\(\rho+\frac1{n^{\nicefrac32}}\)} \\
&\le 
\(\frac1{2n}\)^{k+\ell}  \(1+\frac2{\sqrt n}\)^{k+\ell} 
n^{\nicefrac{3\ell}2} \prod_{\nu=1}^k\frac1{1-\frac{\nu}{2n}\(1+\frac2{\sqrt n}\)}.
\end{align*} 
Thanks to~\eqref{eq:M}, we get  
\begin{align*}
\lim_{z\to\rho} \frac{M'(z)}{A'(z)} 
& \leq \rho^{r+2} \binom{n-p}{k+r}\binom{k+r}{k} 2^{k+r} k!r!
\frac{f^{(\ell)}(\rho)}{(\ell-1)!} \lim_{z\to\rho}\frac{K'(z)}{A'(z)}  \\
&\le \ell
\(\frac1{2n}\)^{k+\ell+r+2} 2^{k+r} (n)_k n^r \(1+\frac2{\sqrt n}\)^{k+\ell} 
n^{\nicefrac{3\ell}2}\prod_{\nu=1}^k\frac1{1-\frac{\nu}{2n}\(1+\frac2{\sqrt n}\)}
\lim_{z\to\rho}\frac{K'(z)}{A'(z)}
 \\
&\le \ell
e^2 2^{-\ell-2}\cdot n^{\nicefrac{\ell}2-2} \(1+\frac2{\sqrt n}\)^{k}
\prod_{\nu=1}^k\frac{1-\frac \nu n}{1-\frac{\nu}{2n}\(1+\frac2{\sqrt n}\)}
\lim_{z\to\rho}\frac{K'(z)}{A'(z)}.
\end{align*} 
The last product is bounded by 1 and if $k=\Ord (\sqrt n)$, then $\(1+\frac2{\sqrt n}\)^{k}$ is
bounded as well. In this case we get 
\begin{equation} \label{MAbound}
\lim_{z\to\rho} \frac{M'(z)}{A'(z)} = \Ord\(\ell 2^{-\ell} n^{\nicefrac{\ell}2-2} \lim_{z\to\rho}\frac{K'(z)}{A'(z)} \)
\end{equation} 
If $k>4\sqrt n$ then it suffices to show that $\(1+\frac2{\sqrt n}\)^k
n^{\nicefrac{3\ell}2}\prod_{\nu=1}^k\frac1{1-\frac{\nu}{2n}\(1+\frac2{\sqrt n}\)}$ is bounded. We
proceed as follows: First observe that 
\[
\frac{1-\frac \nu n}{1-\frac{\nu}{2n}\(1+\frac2{\sqrt n}\)} < 1-\frac{\nu}{2n}\(1-\frac2{\sqrt n}\)
\]
and thus we can write 
\begin{align*} 
\(1+\frac2{\sqrt n}\)^{k}
\prod_{\nu=1}^k\frac{1-\frac \nu n}{1-\frac{\nu}{2n}\(1+\frac2{\sqrt n}\)} &= \(1+\frac2{\sqrt
n}\)^{4\lfloor\sqrt n\rfloor} \prod_{\nu=1}^{4\lfloor\sqrt n\rfloor} \frac{1-\frac \nu
n}{1-\frac{\nu}{2n}\(1+\frac2{\sqrt n}\)} 
\\ 
&\qquad\times  
\(1+\frac2{\sqrt n}\)^{k-4\lfloor\sqrt n\rfloor} 
\prod_{\nu=4\lfloor\sqrt n\rfloor+1}^{k} \frac{1-\frac \nu n}{1-\frac{\nu}{2n}\(1+\frac2{\sqrt
n}\)} \\ 
&\le e^8  \left[\(1+\frac2{\sqrt n}\)
\(1-\frac{2}{\sqrt n}\(1-\frac2{\sqrt n}\)\)\right]^{k-4\lfloor\sqrt n\rfloor} \\
&= e^8  \(1+\frac8{n^{\nicefrac32}}\)^{k-4\lfloor\sqrt n\rfloor} = \Ord(1).
\end{align*} 

Finally, if we let $n$ tend to infinity and apply Lemma~\ref{Klem}, we get 
\[\mu_n(\mathcal{M}_{k,\ell,r}^{\Gamma}) = \mathcal{O}\left(r\ell^2 2^{r-\frac{3\ell-1}2}
n^{-\nicefrac32}\right) = \mathcal{O}\left(\frac1{n^{\nicefrac32}}\right).\qedhere\]
\end{proof}

\begin{lem}
\label{lem:small_k}
The limiting ratio of trees with fewer than $n^{\nicefrac14}$ different labels appearing on the first level leaves is $\mathcal{O}\left(\frac1{\sqrt{n}}\right)$.
\end{lem}

\begin{proof}
The generating function of trees with exactly $k$ different labels appearing on the first level (with no occurrence of a variable and its negation) is given by
\[G_k(z) = \begin{pmatrix}n\\k\end{pmatrix}2^k k! z \prod_{m=0}^{k}\frac{z}{1-mz-\hat{B}(z)},\]
and therefore, their limiting ratio is given by
\begin{align*}
\lim_{z\to\rho} \frac{G'_k(z)}{A'(z)}
&= \begin{pmatrix}n\\k\end{pmatrix}2^k k! \rho
\prod_{m=0}^{k}\frac{\rho}{1-m\rho-\hat{B}(\rho)} \sum_{m=0}^k \frac1{1-m\rho-\hat{B}(\rho)} \cdot
\lim_{z\to\rho} \frac{\hat{B}'(z)}{A'(z)}\\
&\leq \frac1{4n} \prod_{m=0}^{k-1} \left(1-\frac{m}{n}\right) \prod_{m=0}^{k}\frac1{1-\frac{m}{2n}-\frac1{\sqrt{2n}}} \left(\sum_{m=0}^k \frac1{1-\frac{m}{2n}-\frac1{\sqrt{2n}}}\right)\\
&\leq \frac1{4n} \left(\frac1{1-\frac{k}{2n}-\frac1{\sqrt{2n}}}\right)^{k+1} (k+1) \frac1{1-\frac{k}{2n}-\frac1{\sqrt{2n}}}
= \mathcal{O}\left(\frac{k}{n}\right).
\end{align*}
Summing up over all $k$ from 1 to $\lfloor n^{1/4}\rfloor$ yields the result. 
%Therefore, the limiting ratio of trees with fewer than $n^{\nicefrac14}$ different labels on the
%first level is equal to 
%\[\sum_{k=0}^{n^{\nicefrac14}} \mathcal{O}\left(\frac{k}{n}\right) =
%\mathcal{O}\left(\frac1{\sqrt{n}}\right).\qedhere\]
\end{proof}

\begin{lem}
\label{lem:big_l}
The limiting ratio of the family of all trees where the root has more than $n^{\nicefrac18}$ non-leaf subtrees is $\Theta\left(\frac{n^{\nicefrac18}}{2^{n^{\nicefrac18}}}\right)$.
\end{lem}

\begin{proof}
The generating function of the trees with exactly $\ell$ non-leaf subtrees is
\[H_{\ell}(z) = 2z\frac{\hat{B}^{\ell}(z)}{(1-2nz)^{\ell+1}}.\]
Therefore, the limiting ratio of trees with exactly $\ell$ non-leaf subtrees is given by
\begin{align*}
\lim_{z\to\rho}\frac{H'_{\ell}(z)}{A'(z)} 
&= \frac1{n} \frac{\ell \hat{B}^{\ell-1}(\rho)}{(1-2n\rho)^{\ell+1}}\cdot \lim_{z\to\rho} \frac{\hat{B}'(z)}{A'(z)}
\sim \frac1{2n} \frac{\ell \left(\frac1{\sqrt{2n}}\right)^{\ell-1}}{\left(\sqrt{\frac2{n}}\right)^{\ell+1}}\\
&\sim \frac1{2n} \frac{\ell \left(\frac1{\sqrt{2n}}\right)^{\ell-1}}{\left(\sqrt{\frac2{n}}\right)^{\ell+1}}
\sim \frac{\ell}{2^{\ell+1}}, \text{ as }\nti.
\end{align*}
This implies that the limiting ratio of trees with more than $n^{\nicefrac18}$ non-leaf subtrees is given by
\[\sum_{\ell\geq n^{\nicefrac18}} \frac{\ell}{2^{\ell+1}} =
\Theta\left(\frac{n^{\nicefrac18}}{2^{n^{\nicefrac18}}}\right).\qedhere\]
\end{proof}

\section{Tautologies}
\label{part:tauto}

This section deals with tautologic trees and states the first statement
of Theorem~\ref{thm:general} about the probability of constant functions. 
This is not only a
particular case, but also the first step to prove the whole theorem. The first subsection is
devoted to prove the lower bound, the second one to prove the upper bound and the last one to
prove that \emph{almost all tautologies have a very simple shape}.

\subsection{A non-negligible family of tautologies.}
\label{alpha}
In this section, we define a set of ``simple'' tautologies and find a lower bound for the limiting
ratio of this family of tautologies.

\begin{df}
A {\bf simple tautology} (resp. a simple contradiction) realized by the variable $x$ is an $\lor$-rooted
(resp. $\land$-rooted) tree such that both of the labels $x$ and $\bar{x}$ appear on the first level of the tree.
The set of simple tautologies will be denoted by $\mathcal{S}_n$, its complement by
$\overline{\mathcal{S}_n}$.
\end{df}

\begin{lem}
\label{lem:alpha}
There exists a constant $\alpha$ such that, for all $n\geq 0$, 
\[0 < \alpha \leq \mathbb{P}_n(\True).\] 
\end{lem}
\begin{proof}
Let $\mc{R}$ be the family of trees such that their root is labelled by $\lor$, containing
at least $\lfloor \sqrt{n}\rfloor$ leaves on the first level and a single $\land$--rooted tree attached to the root.
Furthermore, the $\lfloor \sqrt{n}\rfloor$ leftmost children of the root are leaves.
The generating function $R(z)$ satisfies
\[R(z) = \frac{z\cdot (2nz)^{\lfloor \sqrt{n}\rfloor}}{1-2nz} \cdot \hat{B}(z) \cdot \frac{1}{1-2nz}.\]
The generating function $R(z)$ has the same dominant singularity as $\hat{A}(z)$ and thus the same as $A(z)$.
By using Lemma~\ref{lem:limiting_ratio}, we get
\[\frac{R'(z)}{A'(z)} = \frac{\rho (2n\rho)^{\lfloor \sqrt{n}\rfloor}}{(1-2n\rho)^2} \cdot \lim_{z\to \rho}\frac{\hat{B}'(z)}{A'(z)} = \frac{e^{-\sqrt{2}}}{8} + o(1).\]
Thus, the family $\R$ contains a positive ratio of all trees. In order to conclude, the last idea
consists of using the birthday paradox (see e.g.~\cite{FGT92} and the references therein)
in order to prove that a positive ratio of trees from $\R$ is obtained by simple tautologies.
The probability that among the $\lfloor \sqrt{n}\rfloor$ leftmost leaves, of the first level, there is no repetition of variables is
$n(n-1)\cdot (n-\lfloor \sqrt{n}\rfloor+1) / n^{\lfloor \sqrt{n}\rfloor}$. Thus
by using Stirling formula this probability tends to $\nicefrac1e$.
Consequently, for $n$ sufficiently large, with probability larger than $\nicefrac12$ there is at least
one repetition of variables among the $\lfloor \sqrt{n}\rfloor$ leftmost leaves of the first level.
Let us take into account the two first (from left to right) occurrences of the same variable,
with probability $\nicefrac12$ they are labelled by two opposite literals, and thus the whole tree is a simple tautology.

Finally, when $n$ tends to infinity, the probability to be a simple tautology is positive. 
Furthermore, the previous results also hold for small value of $n$, which concludes the proof.
\end{proof}

The following investigations aim at deriving a 
numerical lower bound for the probability of the set of simple tautologies. 
In order to get such a result, we utilized a computer algebra software to perform some routine
calculations. In these instances, we will omit the details.

Let $\E^k$ be the set of trees rooted by $\lor$, with exactly $k$ leaves at depth~1 and at most 5
non-leaf subtrees.
Let us take $k$ in $M:=\{\lfloor\sqrt{n}\rfloor,\dots, 15\lfloor\sqrt{n}\rfloor\}$. 
For two generating functions $f, g \in \IR [[z]]$, we write $f \prec g$ (resp. $f \succ g$) if $[z^{r}]f \leqslant [z^{r}]g$ (resp. $[z^{r}]f \geqslant [z^{r}]g$)
for all $r \in \IN$.
The generating function of $\E = \bigcup_{k\in M} \E^k$ satisfies
\[E(z) \succ \sum_{j=1}^5\sum_{k=\lfloor\sqrt{n}\rfloor}^{15\lfloor\sqrt{n}\rfloor} z \cdot (2nz)^k \cdot \frac{(k+1)^j \hat{B}^j(z)}{j!}.\]
The latter generating function is not exactly $E(z)$ because of the approximation 
$(k+1)\cdots(k+j) \sim (k+1)^j$. Using the approximated form to compute the limiting ration, we
obtain after some computations: $\mu_n(\E) \geqslant 0.36618 + o(1)$ for $n$ tending to infinity.

Let us now introduce two subfamilies of $\E^k$. The first family $\E^k_1$ contains all trees of
$\E^k$ 
such that
\begin{enumerate}
\item[(a)] the set $\L$ of all labels appearing among the $\lfloor\sqrt{n}\rfloor$ 
leftmost leaves at depth~1 is of cardinality at least~$\nicefrac{\sqrt{n}}{2}$, 
\item[(b)] and the $k-\lfloor\sqrt{n}\rfloor$ rightmost leaves of depth~1 are labelled by literals whose negation
does not belong to $\L$.
\end{enumerate}
The family $\E^k_2$ contains all trees of $\E^k$ such that
\begin{enumerate}
\item[(a')] the set $\L'$ of all labels appearing among the $\lfloor\sqrt{n}\rfloor$ leftmost leaves at depth~1 is of cardinality at most $\nicefrac{\sqrt{n}}{2}$
\item[(b)] and the $k-\lfloor\sqrt{n}\rfloor$ rightmost leaves of depth~1 are labelled by literals whose negation
does not belong to $\L'$.
\end{enumerate}

The following lemma comes from the fact that a tautology which is not simple fulfills either
condition (a) or (a') and always satisfies condition (b) since a variable and its negation cannot appear in the first level of the tree. 
Therefore, a non-simple tautology is either in $\E^k_1$ or in $\E^k_2$. The same applies to trees
which do not even represent tautologies. Thus we obtain the following result: 
\begin{lem}
We have $\E \cap \overline{\mathcal{S}_n} \subseteq \bigcup_{k\in M} (\E^k_1 \cup \E^k_2)$.
\end{lem}

Let us now estimate the sizes of the sets $\E^k_1$ and $\E^k_2$. We get an upper bound by
multiple counting: Let $E^k_1(z)$ (resp. $E^k_2(z)$) be the generating functions of $\E^k_1$ (resp.
$\E^k_2$). Then 
\begin{equation} \label{Eone}
E^k_1(z) \prec \sum_{j=1}^5 z \cdot (2nz)^{\sqrt{n}} \cdot \left(\left(2n-\frac{\sqrt{n}}{2}\right)z\right)^{k-\sqrt{n}} \cdot \frac{(k+5)^j\hat{B}^j(z)}{j!}.
\end{equation} 
Some multiple counting is done for the leftmost $\lfloor\sqrt{n}\rfloor$ leaves at depth~1, 
because we did not restrict to labellings such that $|\mathcal L|\geq \lceil\nicefrac{\sqrt{n}}2\rceil$. 
Further over-counting appears by forbidding only
$\lceil\nicefrac{\sqrt{n}}2\rceil$ labels instead of $|\mathcal L|$ for the rightmost $k-\lfloor\sqrt n\rfloor$ leaves. 
Using the function on the right-hand side of
\eqref{Eone} for the computation of the limiting ratio, we get $0.24457$ 
as an asymptotic upper bound for the limiting ratio of the family $\bigcup_{k\in
M}\E^k_1$. 

Furthermore, 
\begin{equation} \label{Etwo}
E^k_2(z) \prec \sum_{j=1}^5 \binom{n}{\lfloor\nicefrac{\sqrt{n}}2\rfloor} 2^{\sqrt{n}/2} \cdot z \cdot \left(\frac{\sqrt{n}}{2}z\right)^{\sqrt{n}} \cdot (2nz)^{k-\sqrt{n}} \cdot
 \frac{(k+5)^j\hat{B}^j(z)}{j!}.
\end{equation} 
Again we do some multiple counting in the same fashion as before. 
The function on the right-hand side of \eqref{Etwo} permits us to conclude that the limiting ratio
of the family $\bigcup_{k\in M} \E^k_2$ is of order $o(1)$, as $\nti$.

\begin{prop}\label{prop:alpha}
The limiting ratio of the set of simple tautologies is greater than $\alpha+o(1)$ where 
$\alpha:=0.12161.$
\end{prop}
\begin{proof}
The previous lemma is equivalent to $\E \cap \mathcal{S} \supseteq \E\setminus \bigcup_{k\in M}
(\E^k_1 \cup \E^k_2)$. Thus, for $\nti$ we have 
\[
\mu_n(\mathcal{S}_n) \ge \mu_n(\mathcal{S}_n\cap \E) \geq \mu_n(\E) -
\left(\sum_{k=\sqrt{n}}^{15\sqrt{n}}\mu_n(\E^k_1) +  \mu_n(\E^k_2)\right) \geq 0.36618 - 0.24457 +
o(1).\qedhere
\]
\end{proof}

\subsection{A non-negligible family of non-tautologies.}
In this section, we derive an upper bound for the probability of the set of tautologies. 

\begin{df}\label{family_G}
Let $\check{\mathcal{G}}_k$ be the family of $\lor$-rooted trees with exactly $k$ leaves on the
first level, labelled by $k$ different literals $\alpha_1,\ldots,\alpha_k$ such that each variable can only appear positive or negative and whose non-leaf subtrees are all contradictions.
\end{df}

Note that an $\lor$-rooted tree in $\check{\mathcal{G}}$ computes the function
$\alpha_1\lor\ldots\lor\alpha_k$, and is therefore neither a tautology nor a contradiction.

\begin{lem}
\label{lem:beta}
There exists a constant $\beta$ such that, for all $n\geq 0$, 
\[\mathbb{P}_n(\True)\leq \beta < \frac12.\] 
\end{lem}

\begin{proof}
The generating function of the tree family defined in Definition~\ref{family_G} is given by
\[
\check{G}_k(z) = \begin{pmatrix}n\\k\end{pmatrix} 2^k k! z^{k+1} \frac1{(1-T(z))^{k+1}},
\]
and its limiting ratio is given by
\[
\lim_{z\to\rho}\frac{\check{G}'(z)}{A'(z)} \sim \begin{pmatrix}n\\k\end{pmatrix} 2^k k! \rho^{k+1}
\frac{k+1}{(1-T(\rho))^{k+1}} \lim_{z\to\rho}\frac{T'(z)}{A'(z)}.\]
But Proposition~\ref{prop:alpha} tells us that $\lim_{z\to\rho}\frac{T'(z)}{A'(z)} \geq \frac{\alpha}{2}$. Moreover, we know that $T(\rho)> 0$ and therefore
$\frac{1}{(1-T(\rho))^{k+1}}> 1$. Therefore,
\begin{align*}
\lim_{z\to\rho}\frac{G'(z)}{A'(z)} 
&\geq \begin{pmatrix}n\\k\end{pmatrix} 2^k k! (k+1) \left(\frac1{2n}\right)^{k+1} \frac{\alpha}{2}\\
&=\alpha \cdot \frac{n(n-1)\ldots(n-k+1)}{n^k} \frac{k+1}{4n}\\
&\geq \alpha \cdot \left(1-\frac{k-1}{n}\right)^{k-1} \frac{k+1}{4n}.
\end{align*}
If $k=\Theta(\sqrt{n})$, then the right-hand side of the above inequality is of order
$\Theta(\nicefrac1{\sqrt{n}})$, and the limiting ratio of $\check{\mathcal{G}}_k$ is thus larger than
$\nicefrac{c_k}{\sqrt{n}}$ where $c_k$ is a positive constant. Thus the limiting ratio of the family
$\bigcup_{k=\lfloor\sqrt{n}\rfloor}^{2\lfloor\sqrt{n}\rfloor} \check{\mathcal{G}}_k$ is bounded
from below by a positive constant $c$. We therefore have proved Lemma~\ref{lem:beta}, and $\beta =
\nicefrac12 - c$.
\end{proof}

\subsection{Almost every tautology is simple.}
In general, it is not easy to describe precisely the shape of a tautologic tree, but according
to our distribution on trees, almost every tautology is ``simple''.

\begin{thm}
\label{thm:st}
Asymptotically almost every tautology is a simple tautology, \emph{i.e.}\
\[
\mathbb{P}_n(\True) = \mu_n(\mathcal{S}_n) + o(1), \text{ as }\nti.
\]
\end{thm}

\begin{proof}%[Proof of Theorem~\ref{thm:st}]
To prove Theorem~\ref{thm:st}, we will consider the family of non-simple tautologies, study the
structure of its elements and show that its limiting ratio tends to zero when $n$ tends to
infinity.

Let us consider $\mathcal{N}_n=\mathcal{T}_n\setminus \mathcal{S}_n$ the set of tautologies which
are not simple. Let $t\in\mathcal{N}_n$, with $\ell$ non-leaf subtrees $A_1,\ldots,A_\ell$ and
$k$ different labels $\alpha_1,\ldots,\alpha_k$ appearing in the first level. Since $t$ is not a
simple tautology, the set $L = \{\alpha_1,\ldots,\alpha_k\}$ cannot contain both a variable
and its negation. For each $i\in\{1,\ldots,\ell\}$, the tree $A_i$ has some of its leaves in its
first level labelled by labels in $L$ and others carrying new labels.
What will be called a ``new variable'' in the following is the label of a leaf in the \emph{first
level} of one of the $A_i$ if it does not belong to $L$. 

Let us show that there exists at least one $i\in\{1,\ldots,\ell\}$ such that $A_i$ has at most
$\ell-1$ new variables. 

Let us assume that for all $i\in\{1,\ldots,\ell\}$, $A_i$ has at least $\ell$ new variables in the
first level. We assign one of those new variables, say $\nu_1$ and belonging to $A_1$, to $\False$.
Then $A_1$ computes $\False$. But, there must exist $A_{(2)}\in\{A_2,\ldots,A_{\ell}\}$ which is
not a contradiction for this assignment, because if $A_1 = A_2 = \ldots = A_\ell \equiv \False$ for
$\nu_1 = \False$, then the whole tree $t$ would compute $\alpha_1\lor\ldots\lor\alpha_k$ and would
thus not be a tautology.

Let us iterate this algorithm: After step $s-1\leq\ell$, we have assigned $\nu_1 = \nu_2 = \ldots =
\nu_{s-1} = \False$ and at least $s-1$ trees of $\{A_1,\ldots A_\ell\}$ compute $\False$. At step 
$s$ note that at least one of the remaining subtrees must still not be a contradiction: Let us
call this tree $A_{(s)}$. It has at least $\ell$ new variables and we have assigned $s-1 \leq
\ell$ variables to $\False$ so far. Therefore we still have free new variables among the new
variables of $A_{(s)}$, and we can assign one, denoted by $\nu_s$, to $\False$. 

After the $\ell^{\text{th}}$ step, we have found an assignment of variables different from those
in $L$ such that all trees of $\{A_1,\ldots,A_\ell\}$ compute $\False$. Thus $t$ computes $\alpha_1\lor\ldots\lor\alpha_k$, which is impossible since $t$ is a tautology.

Therefore, there exist at least one $i\in\{1,\ldots,\ell\}$ such that $A_i$ has at most $\ell-1$ new
variables in the first level. 

It means that $\mathcal{N}_n \subseteq \bigcup_{k=0}^n \bigcup_{\ell=0}^{\infty}
\bigcup_{r=0}^{\ell-1} \mathcal{M}^{\emptyset}_{k,\ell,r}$ (\emph{cf.}\ Section~\ref{part:useful}).
Let us decompose this union into three distinct unions:
\begin{align*} 
&\bigcup_{k=0}^n \bigcup_{\ell=0}^{\infty} \bigcup_{r=0}^{\ell-1}\mathcal{M}^{\emptyset}_{k,\ell,r}
\\
&\qquad =\left(\bigcup_{k=0}^{\lfloor n^{\nicefrac14}\rfloor} \bigcup_{\ell=0}^{\infty} \bigcup_{r=0}^{\ell-1}
\mathcal{M}^{\emptyset}_{k,\ell,r}\right)
\cup \left(\bigcup_{k=\lfloor n^{\nicefrac14}\rfloor}^{n} \bigcup_{\ell=\lfloor
n^{\nicefrac18}\rfloor}^{\infty}
\bigcup_{r=0}^{\ell-1} \mathcal{M}^{\emptyset}_{k,\ell,r}\right)
\cup \left(\bigcup_{k=\lfloor n^{\nicefrac14}\rfloor}^{n}
\bigcup_{\ell=0}^{\lfloor n^{\nicefrac18}\rfloor}
\bigcup_{r=0}^{\ell-1}\mathcal{M}^{\emptyset}_{k,\ell,r}\right).
\end{align*} 
Thanks to Lemma~\ref{lem:small_k}, the first term has a limiting ratio tending to zero as $n$
tends to infinity; Lemma~\ref{lem:big_l} guarantees that the second term has also a limiting ratio
tending to zero, and by Lemma~\ref{lem:Mklr}, the third term satisfies
\[
\mu_n\left(\left(\bigcup_{k=\lfloor n^{\nicefrac14}\rfloor}^{n} \bigcup_{\ell=0}^{\lfloor
n^{\nicefrac18}\rfloor}
\bigcup_{r=0}^{\ell-1}\mathcal{M}^{\emptyset}_{k,\ell,r}\right)\right)
= \mathcal{O}\left((n-n^{\nicefrac14})(n^{\nicefrac18})^2\frac1{n^{\nicefrac32}}\right) =
\mathcal{O}\left(\frac1{n^{\nicefrac14}}\right)
\]
and is thus also tending to zero when $n$ tends to infinity. Thus, $\mu_n(\mathcal{N}_n\setminus
\mathcal{S}_n) = o(1)$ and Theorem~\ref{thm:st} is proved.
\end{proof}

\section{Probability of functions larger than a fixed $f_0$}
\label{part:larger}

In this section we investigate the probability of all the functions that are ``larger'' than a fixed
given Boolean function. The result will enable us to prove the second statement of
Theorem~\ref{thm:general} about literal functions, proved in Section~\ref{part:simple_x}.

\begin{df}
\label{df:larger}
Let $f$ and $g$ be two Boolean functions of $n$ variables, we say that $g\geq f$ if and only if, $g(x_1,\ldots,x_n) \geq f(x_1,\ldots,x_n)$ for all $(x_1,\ldots,x_n)\in \{0,1\}^n$.
\end{df}
In the following we fix a non-constant Boolean function $f_0$ and estimate $\mathbb{P}_n(f\geq f_0)$. 
We first consider the case $f_0 = x_1\land \ldots \land x_p$ and then generalize the result to any
Boolean function. 

\begin{prop}
\label{prop:comparison}
Choose $n_0$ and a non-constant function $f_0\in\mathcal{F}_{n_0}$. Moreover, denote the natural
extension of $f_0$ to $\mathcal{F}_n$ for $n>n_0$, \emph{i.e.} the function
$(x_1,\dots,x_{n_0},x_{n_0+1},\dots,x_n)\mapsto f_0(x_1,\dots,x_{n_0})$, by $f_0$ as well. 
Denote by $F$ a random Boolean function having law $\mathbb P_n$, 
\emph{i.e.} $\mathbb P(F= f) = \mathbb P_n(f)$ for all Boolean function $f$.
Then 
\[
\mathbb{P}(F\geq f_0)  = \mathbb{P}_n(\True)\ (1+o(1)),  \text{as $n$ tends to infinity.}
\]
\end{prop}

The remaining part of this section is devoted to the proof of Proposition~\ref{prop:comparison}. 
We prove this proposition in two steps: First, assume that $f_0$ is a conjunction of literals, and
then we extend the proof to obtain the general result. 

\subsection{Case I: $\bs{f_0}$ is a conjunction of literals}

Let  $f_0 = \gamma_1\land \ldots \land \gamma_p$, where the $\gamma_i$'s are literals.
First, let us remark that $\mathbb{P}(F\geq f_0) \geq \mathbb{P}_n(\True) \geq \alpha >0$.
Let us consider an associative tree $t$ computing a Boolean function which is larger than $f_0$
and not a tautology.

The family of trees with no leaf on the first level has a limiting ratio which is asymptotically
equal to $\nicefrac{1}{n\sqrt{2n}}$ when $n$ tends to infinity (\emph{cf.}\ 
Proposition~\ref{prop:noleaf}). It is thus negligible compared with $\mathbb{P}(\True)$. 
Thus, we can assume that $t$ has at least one leaf on the first level.

Consider first the case where $t$ is rooted by an $\lor$:
\begin{itemize}
\item Let us first assume that there exists one leaf on the first level of $t$, labelled by one of
the $\gamma_i$. The family of trees with at least one leaf on the first level and with a label from
the set $\Gamma = \{\gamma_1,\ldots,\gamma_p\}$ has a limiting ratio equivalent to $p\sqrt{\nicefrac{2}{n}}$,
in view of Proposition~\ref{prop:A_alpha}. 
%\[G_{p}(z) = 2nz + 2z\cdot pz \frac{1}{(1-\hat{B}(z)-(2n-p)z)(1-\hat{A}(z))} - 2pz^2.\]
%Its limiting ratio is given by
%\[\lim_{z\to \rho}\frac{G_p'(z)}{A'(z)} \sim p\sqrt{\frac{2}{n}}\]
The limiting ratio of such trees is thus negligible compared with the limiting ratio of the set of
tautologies. We can thus neglect this family.

\item Let us assume that $t$ has no leaf on the first level labelled by a literal chosen in
$\Gamma = \{\gamma_1,\ldots,\gamma_p\}$. Let us denote by $k$ the number of different labels, denoted by
${\alpha_1,\ldots,\alpha_k}$, appearing on the first level of $t$ and by $\ell$ the number of its
non-leaf subtrees, denoted by $A_1,\ldots,A_{\ell}$. Observe that since $t$ is not a tautology,
the labels appearing on the first level of $t$ cannot contain a variable and its negation. The
subtrees $A_1,\ldots, A_{\ell}$ have themselves leaves on their first level (\emph{i.e.}\ on the second
level of $t$), and those leaves are labelled either by ``old variables'', \emph{i.e.}\ by literal chosen
from $\Ord = \{\alpha_1,\ldots,\alpha_k,\gamma_1,\ldots,\gamma_p\}$ and their negations, or by
``new variables'', i.e by other literals. Assume that for all $i\in\{1,\ldots,\ell\}$, $A_i$ has at
least $\ell$ different new variables appearing on its first level. Thus, since each $A_i$ is
rooted by an $\land$, we can find an assignment of the variables $\{x_1,\ldots,x_n\}\setminus
\Ord$ such that all $A_i$ compute $\False$ for this assignment. Assign then all
the leaves on the first level of $t$ to $\False$ and $\gamma_1,\dots,\gamma_p$ to $\True$. Then $t$ computes
$\False$ while $\gamma_1\land \ldots\land \gamma_p$ takes the value $\True$ for this assignment. 
But this is impossible!

Thus, there exists at least one $A_i$ which has fewer than $\ell$ new variables on its first level.
This means that $t$ belongs to the set $\bigcup_{r=0}^{\ell-1}\mathcal{M}^{\Gamma}_{k,\ell,r}$.

The limiting ratio of such trees is thus less than the limiting ratio of $\bigcup_{k,\ell\geq
0}\bigcup_{r=0}^{\ell-1}\mathcal{M}^{\Gamma}_{k,\ell,r}$, and thanks to the results proved
Section~\ref{part:useful}, we have 
\[
\mu_n\left(\bigcup_{k,\ell\geq 0}\bigcup_{r=0}^{\ell-1}\mathcal{M}^{\Gamma}_{k,\ell,r}\right) 
%\sim \mu_n\left(\left(\bigcup_{k=\lfloor n^{\nicefrac14}\rfloor}^{n} \bigcup_{\ell=0}^{\lfloor
%n^{\nicefrac18}\rfloor}
%\mathcal{M}^p_{k,\ell,\ell}\right)\right)
= \mathcal{O}\left((n-n^{\nicefrac14})(n^{\nicefrac18})^2\frac1{n^{\nicefrac32}}\right) =
\mathcal{O}\left(\frac1{n^{\nicefrac14}}\right)
\]
and this family is also negligible in in front of tautologies. 
\end{itemize}

If $t$ is rooted by an $\land$, then its first level leaves have labels chosen from the set $\Gamma = \{\gamma_1,\ldots,\gamma_p\}$. The family of trees with first level leaves labelled in $\{\gamma_1,\ldots,\gamma_p\}$ has generating function
\[H_{\Gamma}(z) = 2nz + 2z\frac{(\hat{B}(z)+pz)^2}{1-\hat{B}(z)-pz}\]
and its limiting ratio is asymptotically equal to $\nicefrac1{n\sqrt{2n}}$. It is thus negligible
compared with the set of tautologies. We have thus proved that
\[\mathbb{P}(F\geq \gamma_1\land \ldots\land \gamma_p) \sim \mathbb{P}_n(\True),\]
as $n$ tends to infinity.

\subsection{Case II: $\bs{f_0}$ is any non-constant Boolean function} 

All Boolean functions $f_0$ can be written as $f_0 = (\gamma_1 \land \ldots \land\gamma_p)\lor
g_0$ for some integer $p\geq 1$, some literals $\{\gamma_1,\ldots,\gamma_p\}$, 
and some Boolean function $g_0$. Thus,
\[\mathbb{P}_n(\True)\leq \mathbb{P}(F\geq f_0) \leq \mathbb{P}(F\geq \gamma_1 \land \ldots
\land \gamma_p) \sim \mathbb{P}_n(\True)\]
and Proposition~\ref{prop:comparison} is proved. \hfill \qedsymbol

\section{Literals}
\label{part:simple_x}
The aim of this section is to estimate the probability of a literal Boolean function \emph{i.e. }a function of the shape
$((x_1,\ldots,x_n)$ $\mapsto x)$ where $x$ is a literal among
$\{x_1,\bar{x}_1,\ldots,x_n,\bar{x}_n\}$, therefore proving the second statement of Theorem~\ref{thm:general}. 
As for tautologies, we will prove that a typical tree computing this function has a very simple shape.

\begin{df}
A tree $t$ is a {\bf simple $x$ tree} if it is rooted by an $\lor$ (resp. $\land$), with one
single leaf on the first level, labelled by $x$ and with one non-leaf subtree which is a tautology
(resp. a contradiction). We denote by $\mathcal{X}$ the family of such trees and and by $X_m$ the
number of simple $x$ trees of size $m$.
\end{df}

\begin{lem}
\label{lem:SX}
For $n$ tending to infinity, the limiting ratio of the set simple $x$ trees satisfies
\[\mu_n(\mathcal{X}) \sim \frac{\mathbb{P}_n(\True)}{n^2}.\]
\end{lem}

\begin{proof}
The generating function associated with the set of simple $x$ trees 
is given by the following generating function:
\[X(z) = 4z^2T(z),\]
where the $4$ factor contains the choice of the label of the root 
and the order of its first and second child.
Therefore,
\[\mu_n(\mathcal{X}) = \lim_{z\to\rho} \frac{X'(z)}{A'(z)}  = 4\rho^2 \lim_{z\to\rho} \frac{T'(z)}{A'(z)}.\]
Observing that $\frac{T'(z)}{A'(z)}$ tends to $\mathbb{P}_n(\True)$ when $z$ tends to $\rho$ and that
$\rho\sim\nicefrac1{2n}$ when $n$ tends to infinity (\emph{cf.}~Proposition~\ref{prop:sing}) permits to
complete the proof.
\end{proof}

\begin{thm}
For $n$ tending to infinity,
\[\mathbb{P}_n(x) \sim \mu_n(\mathcal{X}).\]
\end{thm}

\begin{proof}
Thanks to Lemma~\ref{lem:SX} and Proposition~\ref{prop:alpha}, we know that
\[\mathbb{P}_n(x) \geq \frac{\mathbb{P}_n(\True)}{2n^2}\geq \frac{\alpha}{2n^2}\]
when $n$ tends to infinity.
Let $t$ be a tree computing $x$. Let us assume that it is rooted by an $\land$ (the case of an
$\lor$-rooted tree would be treated in the very same way). The family of trees computing $x$ with no
leaf on the first level has the same limiting ratio for all
$x\in\{x_1,\bar{x}_1,\ldots,x_n,\bar{x}_n\}$. Therefore, if we denote this family by
$\mathcal{A}^{(0)}_x$, its limiting ratio satisfies (\emph{cf.}~Proposition~\ref{prop:noleaf})
\[2n\mu_n(\mathcal{A}^{(0)}_x) \leq \mu_n(\mathcal{A}^{(0)}) \sim \frac1{n\sqrt{2n}}.\]
Thus, the limiting ratio of trees with no leaf on the first level, computing $x$ has order
$\mathcal{O}(n^{-\nicefrac52})$. This family is negligible in comparison with the family of simple
$x$ trees. Thus we can focus on the family $\mathcal T_x$ of trees having at least one leaf on the
first level.

The leaves on the first level of $t$ have to be labelled by $x$ since $t$ computes $x$. And the
non-leaf subtrees calculate functions that are larger than $x$ (in the sense of
Defintion~\ref{df:larger}). Thus, a tree computing $x$ is almost surely an $\land$-rooted (resp. $\lor$-rooted) tree with leaves on the first level labelled by $x$ and with non-leaf subtrees larger than $x$ (resp. smaller than $x$).

Let us denote by $L_x(z)$ the generating function of trees larger than $x$. In view of Proposition~\ref{prop:comparison}, we have 
\[\lim_{z\to\rho}\frac{L'_x(z)}{A'(z)} \sim \mathbb{P}_n(\True), \text{ when } \nti.\]
Note also that, by symmetry, the family of trees smaller than $x$ has the same generating
function. Thus, the above described family of trees has the same limiting ratio as that of all
trees computing $x$ and its generating function is given by
\[T_x(z) = \frac{2z^2}{(1-(L_x(z)-z))(1-L_x(z))} - 2z^2. \]
The limiting ratio is then 
\[
\mu_n(\mathcal{T}_x) 
\sim 
2\rho^2 \left(\frac1{(1-(L_x(\rho)-\rho))^2(1-L_x(\rho))} + \frac1{(1-(L_x(\rho)-\rho))(1-L_x(\rho))^2}\right)
\lim_{z\to\rho}\frac{L'_x(z)}{A'(z)} 
\sim 4\rho^2\ \mathbb{P}_n(\True), 
\]
as $\nti$. Thus,
\[
\mathbb{P}_n(x) \sim \mu_n(\mathcal{X}) \sim \frac{\mathbb{P}_n(\True)}{n^2}, \text{ as }\nti.
\qedhere
\]
\end{proof}

\section{General case: minimal trees and expansions.}
\label{part:general}

In this last section we prove the last statement of Theorem~\ref{thm:general}. We use different
\emph{expansions} of trees, as it was done in other random Boolean tree models (\emph{cf.}~\cite{FGGG12}
for implication random trees and \cite{GGKM15} for and/or trees). The first subsection defines the
expansions, the second subsection states an asymptotic lower bound for $\mathbb{P}_n(f)$, and the
third subsection states an asymptotic upper bound and thus completes the proof of
Theorem~\ref{thm:general}.

\subsection{Expansions}
\begin{df}
Let $t$ be an associative tree. The tree given by adding a new subtree $t_e$ to an internal node $\nu$ of $t$ is called an {\bf expansion} of $t$.
An expansion is {\bf valid} if the expanded tree computes the same function as $t$.
\begin{itemize}
\item The expansion is called a {\bf tautology expansion} (resp. a contradiction expansion) if the
added tree $t_e$ is a tautology (resp. a contradiction) and if $\nu$ is labelled by a $\land$
(resp. $\lor$). Obviously, such an expansion is valid.
\item It is called a {\bf $\boldsymbol{\hat{B}}$-expansion} if the added tree $t_e$ is not a single leaf.
\end{itemize}
\end{df}

Given a family of trees $\mathcal{T}$, we denote by $E(\mathcal{T})$ the set of trees obtained by
a single tautology expansion of a tree in $\mathcal{T}$, by $E^k(\mathcal{T})$ the set of trees
obtained by $k$ successive tautology expansions done at (not necessarily distinct) vertices of a
tree in $\mathcal{T}$, and by $E^{\geq k}(\mathcal{T})$ the set of all trees obtained by at least
$k$ successive tautology expansions done at (not necessarily distinct) vertices of a tree in
$\mathcal{T}$. Finally, we set $E^*(\mathcal{T}):=\bigcup_{k\ge 1} E^k(\mathcal{T})$. 

If the considered expansions are $\hat{B}-$expansions, we change the above notation by 
replacing $E$ by $E_{\hat{B}}$.

\begin{rmq}
Whatever type of expansion (tautology or $\hat{B}$) we consider, note that nesting expansions
(adding $t_e$ to $t$, then expanding $t_e$, and so on) does not generate new structures, since
this can always be realized by a single expansion. Therefore, requiring that the expansions are
done at the vertices of the original tree is no restriction. 
\end{rmq}

\begin{rmq}
\label{rmk:inclusions}
For every family $\mathcal{T}$ of trees the inclusion $E^*(\mathcal{T}) \subseteq
E_{\hat{B}}^*(\mathcal{T})$ holds.
\end{rmq}

\subsection{Tautology expansions}
\begin{prop}
\label{prop:several_exp}
For $n$ tending to infinity, the limiting ratio of $E^*(\mathcal{M}_f)$ is asymptotically equal to the limiting
ratio of $E(\mathcal{M}_f)$. Furthermore, 
\[\mu_n(E(\mathcal{M}_f)) = \Theta\left(\frac1{n^{L(f)}}\right), \text{ as }\nti.
\]
Since every tree in $E^*(\mathcal{M}_f)$ computes $f$, this implies 
\[\mathbb{P}_n(f) \geq \mu_n(E^*(\mathcal{M}_f)) = \Theta\left(\frac1{n^{L(f)}}\right).\]
\end{prop}

\begin{proof}
Let $\Phi_k(z)$ be the generating function of $E^k(\mathcal{M}_f))$. Given a tree $t_e$, the number of places where $t_e$ can be added to a given minimal tree $t$ is 
\[P_t = \sum_{i \text{ internal node of } t} (d(i) + 1)\]
where $d(i)$ is the number of children of the internal node $i$. Let $i_t$ denote the number of
internal nodes of $t$ and $|t|$ the size of $t$. Then 
$P_t = i_t + |t| - 1$. Since $t$ is minimal (\emph{i.e.}\ $|t|=L(f)$) and since $1 \leq i_t\leq
\lfloor\frac{L(f)}{2}\rfloor$, we have $L(f) \leq P_t \leq \frac{3L(f)}{2}$ which yields 
\[
m_f z^{L(f)} L(f) T(z) \prec \Phi_1(z) \prec m_f z^{L(f)} \frac{3L(f)}{2} T(z),
\]
and thus 
\begin{equation} \label{Schranken}
m_f L(f) \rho^{L(f)} \lim_{z\to\rho} \frac{T'(z)}{A'(z)}\leq \mu_n(E(\mathcal{M}_f)) \leq
m_f\frac{3L(f)}{2} \rho^{L(f)} \lim_{z\to\rho} \frac{T'(z)}{A'(z)}. 
\end{equation} 
From Section~\ref{part:tauto} we know that $0<\alpha\le \mathbb{P}_n(\True) = \lim_{z\to\rho}
\frac{T'(z)}{A'(z)}\le \beta$ and since $\rho\sim\frac1{2n}$ when $n$ tends to infinity, we get  
\[\mu_n(E(\mathcal{M}_f)) = \Theta\left(\frac1{n^{L(f)}}\right).\]
If we do $k$ successive expansions in a minimal tree, we have at most $\lfloor\nicefrac{3
L(f)}2\rfloor$ different places for the first one, $\lfloor\nicefrac{3 L(f)}2\rfloor+1$ for the
second one, and so on. We thus have the following inequality: 
\[
\Phi_k(z) \prec m_f z^{L(f)} \begin{pmatrix}\lfloor\nicefrac{3 L(f)}2\rfloor+k-1\\k\end{pmatrix} T(z)^k
\]
and thus
\begin{align*}
\mu_n(E^k(\mathcal{M}_f)) = \lim_{z\to\rho}\frac{\Phi'_k(z)}{A'(z)} 
&\leq m_f \rho^{L(f)} \begin{pmatrix}\lfloor\nicefrac{3 L(f)}2\rfloor+k-1\\k\end{pmatrix} k
T(\rho)^{k-1} \lim_{z\to\rho}\frac{T'(z)}{A'(z)}\\
%&= m_f \rho^{L(f)} \begin{pmatrix}\lfloor\nicefrac{3 L(f)}2\rfloor+k-1\\k\end{pmatrix} 
%kT(\rho)^{k-1} \mathbb{P}_n(True) \\
&\leq \beta m_f \rho^{L(f)} \binom{\lfloor\nicefrac{3 L(f)}2\rfloor+k-1}k
kT(\rho)^{k-1}, 
\end{align*}
for all $n\ge 0$. Hence 
\begin{align*} 
\mu_n(E^{k\geq 2}(\mathcal{M}_f)) 
%&\leq \beta m_f \rho^{L(f)} \sum_{k\geq 2}\binom{\lfloor\nicefrac{3 L(f)}2\rfloor+k-1}k 
%kT(\rho)^{k-1} \\
&\le\beta m_f \rho^{L(f)} \lfloor\nicefrac{3
L(f)}2\rfloor\left(\frac{1}{(1-T(\rho))^{\lfloor\nicefrac{3 L(f)}2\rfloor+1}}-1\right)
\end{align*} 
where we used $\sum_{k\geq 2} \binom{C+k-1}{k}kz^{k-1}= \frac{C}{(1-z)^{C+1}}-C$ which is an
immediate consequence of 
\begin{equation}
\label{eq:binome}
\sum_{k\geq 0} \binom{C+k-1}{k}z^k = \frac1{(1-z)^C}.
\end{equation}
Since $T(\rho)\leq \hat{B}(\rho)\leq \nicefrac1{\sqrt{2n}}$ (a tautology cannot be a single leaf)
and $\frac{C}{(1-z)^{C+1}}-C=O(z)$, we obtain
\[\mu_n(E^{\geq 2}(\mathcal{M}_f)) = \mathcal{O}\left(\frac1{n^{L(f)+\nicefrac12}}\right).\]
By \eqref{Schranken} the same calculations yield the lower bound.
\end{proof}

\subsection{Irreducible trees}

\begin{df}
Let $t$ be a tree computing $f$. If $t$ cannot be obtained by a tautology expansion of a smaller tree computing $f$, then $t$ is called {\bf irreducible}. We denote by $\mathcal{I}_f$ the set of irreducible trees of $f$ which are not minimal trees of $f$.
\end{df}

Take a tree computing $f$ and simplify it according to tautology expansions until it is
irreducible. The simplified tree is either in $\mathcal{M}_f$ or in $\mathcal{I}_f$. Thus
\begin{equation}
\label{eq:irr}
\mathbb{P}_n(f) \leq \mu_n(E^*(\mathcal{M}_f)) + \mu_n(E^*(\mathcal{I}_f)).
\end{equation}

\begin{prop}
\label{prop:irreductibles}
We have the following asymptotic result: $\mu_n(E^*(I_f)) = o\left(\nicefrac1{n^{L(f)}}\right)$.
\end{prop}

To prove this proposition, we have to understand better the shape of an irreducible tree of $f$. Let $t$ be such a tree. Let us ``simplify'' the tree as follows: 
\begin{itemize}
\item assign all leaves of $t$ which are labelled by inessential variables\footnote{A variable $x$
is an inessential variable of $f$ if the restriction of $f$ on the subset $\{x=\True\}$ is equal to its restriction on the subset $\{x=\False\}$.} 
of $f$ to $\True$, and then 
\item simplify the tree as follows: As soon as a leaf is assigned to $\True$ (resp. $\False$) and
its parent is $\land$ (resp. $\lor$), we cut the leaf. If its parent is $\lor$ (resp. $\land$), we
cut the subtree rooted at this $\lor$ (resp. $\land$), \emph{i.e.}\ at the parent.
\end{itemize}
The obtained tree, denoted by $t^{\star}$ contains no inessential variable and still computes $f$.
It possibly has internal nodes with a single child (called unary nodes), and connectives labelling
some leaves (where $\lor$-leaves compute $\False$ and $\land$-leaves compute $\True$, the other
leaves being labelled by essential variables of $f$. But, since this tree computes $f$, it can be
seen that its size cannot be smaller than $L(f)$. Indeed, a tree with unary nodes and leaves
labelled by connectives can be simplified such that we obtain a proper and/or tree that still computes $f$,
\emph{i.e.}\ with at least $L(f)$ nodes, and this simplification process reduces the number of nodes.
The tree $t^{\star}$ belongs to the following family of trees:
\begin{df}
Let us denote by $\mathcal{S}$ the set of trees with internal nodes labelled by $\land$ and $\lor$
in a stratified way and with leaves labelled by essential variables of $f$ (or their negations) or $\land$ and $\lor$
(again in compliance with the stratification), in which internal nodes can have one child or more. We denote by $S_{\ell}$ the number of such trees of size $\ell$ (the size being the total number of nodes of the tree).
\end{df}

\begin{rmq}
The number $S_{\ell}$ only depends on $\gamma$, the number of essential variables of $f$, and on
$\ell$. 
\end{rmq}

Note that during the simplification process, we have cut either leaves or non-leaf subtrees. We will
prove in the following lemmas that the family of trees such that the simplification process cuts
no non-leaf tree and the family of trees in which we have cut ``many'' single leaves is negligible
in comparison with the family of trees computing $f$ (the limiting ratio of which is at least of
order $\nicefrac1{n^{L(f)}}$).

\begin{lem}
Let $\Gamma\subseteq \{x_1,\ldots, x_n\}$ be of cardinality $\gamma$ and set
$Y=\{x_1,\ldots,x_n\}\setminus \Gamma$. Moreover, let $\mathcal{N}_{\gamma}$ be the family of
trees such that no node labelled by $\lor$ (resp. $\land$) has a leaf labelled by a positive
(resp. negated) variable from $Y$ as a child. Then $\mu_n(\mathcal{N}_{\gamma})=0$,
when $n$ is large enough.
\end{lem}

Note that the family $\mathcal{N}_{\gamma}$ contains the family of trees computing a function $f$
having $\Gamma$ as its set of essential variables and such that the simplification process only
cuts leaves and no non-leaf tree. 

\begin{proof}
The family $\mathcal{N}_{\gamma}$ has the same limiting ratio as associative trees in which leaves
are labelled by literals from a set of cardinality $2n-(n-\gamma) = n+\gamma$. Therefore, the
singularity $\nu_n$ of the generating function of this family is of squareroot type and satisfies 
$\nu_n \sim \nicefrac1{n+\gamma}$ and is thus strictly larger than $\rho_n\sim \nicefrac1{2n}$ 
for large enough $n$. This implies the assertion.
\end{proof}

The following lemma ensures that we have cut only a few ``single'' leaves:
\begin{lem}\label{lem:few_leaves}
Let $\ell$ be an integer and $\Gamma\subseteq\{x_1,\ldots, x_n\}$ be of cardinality $\gamma$. Set 
$Y=\{x_1,\ldots,x_n\}\setminus \Gamma$ and let us denote by $\mathcal{N}_{\ell}$ the family of
trees with at least $\ell$ leaves labelled by variables from $Y$ (or their negations) such that 
none of these leaves has an ancestor labelled by $\lor$ (resp. $\land$) which has a child being a leaf labelled by a positive (resp. negated) variable from $Y$. 
Then $\mu_n(\mathcal{N}_{\ell}) = o\left(\nicefrac1{n^{\ell+1}}\right)$.
\end{lem}

Note that the family $\mathcal{N}_{\ell}$ contains the family of trees computing a function $f$
having $\Gamma$ as its set of essential variables and in which the simplification process cuts at
least $\ell$ single leaves. 

\begin{proof}
Let us consider the family of trees obtained as follows (see Figure~\ref{fig:construction}):
\begin{enumerate}
\item Take a rooted unlabelled tree $t_0$ having $\ell$ leaves and no nodes of arity 1, with 
$\ell+1\leq |t_0|\leq 2\ell-2$;
\item add to each internal node some subtrees (with internal nodes unlabelled and leaves labelled by literals)
which are not single leaves labelled by a variable from $Y$ or its negation,
\item replace each edge by a sequence of internal (unlabelled) nodes with subtrees; 
(with internal nodes unlabelled and leaves labelled by literals) attached to them 
which are not single leaves labelled by a positive (resp. negated) variable from $Y$ if their parent is
$\lor$ (resp. $\land$);
\item choose a label ($\land$ or $\lor$) for the root and deduce the labels of all internal nodes, knowing that a node and its child cannot have the same label (stratification property);
\item finally, replace each leaf of $t_0$ by a tree rooted by $\land$ (resp. $\lor$ according to
the stratification) with at least one literal from $Y$ (resp. from the negations of $Y$) and 
no literal from the negations of $Y$ (resp. from $Y$) on the first level.
\end{enumerate}

\begin{figure}
\begin{center}
\fbox{\includegraphics[width=.35\textwidth]{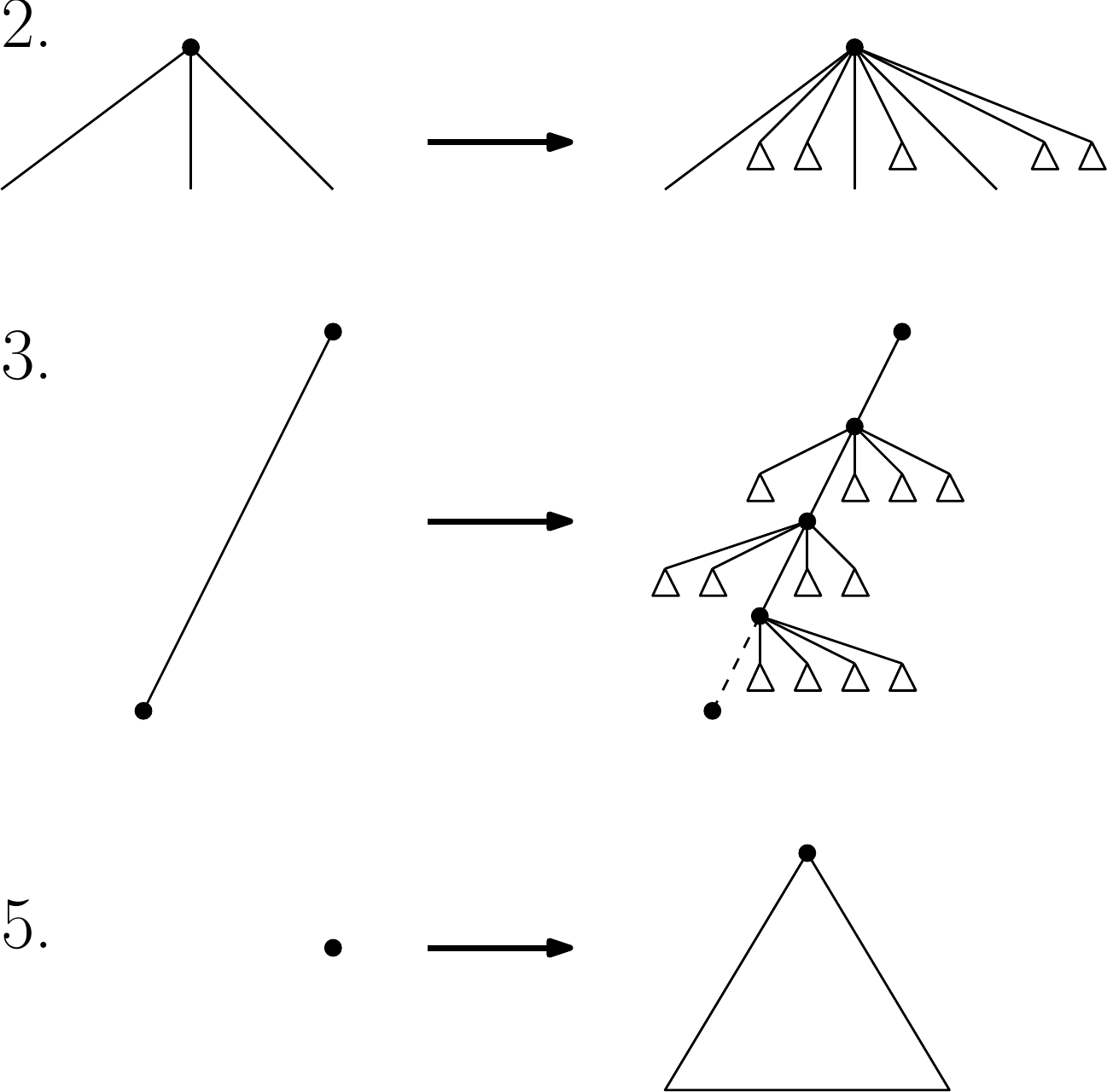}}
\end{center}
\caption{Proof of Lemma~\ref{lem:few_leaves}.}
\label{fig:construction}
\end{figure}
The obtained family contains $\mathcal{N}_{\ell}$ and its generating function is given by
\[
F(z) = 2 C_{\ell} \sum_{r=\ell}^{2\ell-3} z^{r+1-\ell} (zX(z))^{\ell} V(z)^{r} W(z)^{2r+1-\ell},
\]
where 
\begin{itemize} 
\item the index $r$ in the summation represents the number of edges of the tree $t_0$ chosen in the
construction; 
\item the factor $C_{\ell}$ is the number of choices for this $t_0$ and the factor 2 for its root 
label;
\item the factor $z^{r+1-\ell}$ marks the internal nodes of $t_0$;
\item the function
\[
zX(z) = \frac{(n-\gamma)z^2}{[1-(\hat{A}(z)-2(n-\gamma)z)][1-(\hat{A}(z)-(n-\gamma)z))]}-(n-\gamma)z^2
\]
is the generating function of the set of trees rooted by $\land$ with at least one literal from
$Y$ and no literal from the negations of $Y$ on the first level;
\item the function
\[
V(z) = \frac{1}{1-z\left(\left(\frac{1}{1-(\hat{A}(z)-(n-\gamma)z)}\right)^2-1\right)}
\]
is the generating function of the sequences of internal nodes that replace the edges of $t_0$: 
It is the generating function of sequences (possibly empty) of one internal node marked by $z$
and two sequences (one on the left and one on the right of the existing edge - that cannot be both empty)
of subtrees being different from a leaf labelled by a positive variable of $Y$; and 
\item the function 
\[
W(z) = \frac1{1-(\hat{A}-(n-\gamma)z)}
\]
is the generating function of the sequences of trees which are different from a single leaf with a label from $Y$.
Note that such sequences are attached in the 3rd step and they can be placed left from every edge
in $t_0$ ($r$ choices) or to the right of the rightmost child of any internal node of $t_0$
($r+1-\ell$ choices).  
\end{itemize}
To estimate the limiting ratio of this family, observe that the singularity of $F$ is $\rho$ and
that it is a squareroot singularity. Therefore, the limiting ratio of this family can be computed
by Lemma~\ref{lem:limiting_ratio}, \emph{i.e.}\ by $\lim_{z\to\rho}\frac{F'(z)}{A'(z)}$. 
To compute this limiting ratio, let us note that 
$\lim_{z\to\rho}\frac1{A'(z)} = 0$ such that many terms in $F'(z)/A'(z)$ can be neglected
(\emph{cf.}\ last paragraph in the proof of Proposition~\ref{prop:noleaf}). 
We obtain:
\begin{align*}
\lim_{z\to\rho}\frac{F'(z)}{A'(z)} = 
2 C_{\ell} \sum_{r=\ell+1}^{2\ell-2} \rho^{r+1} 
\Bigg(
&\ell X(\rho)^{\ell-1}V(\rho)^{r}W(\rho)^{2r+1-\ell}\lim_{z\to\rho}\frac{X'(z)}{A'(z)}\\
&+ r X(\rho)^{\ell}V(\rho)^{r-1}W(\rho)^{2r+1-\ell}\lim_{z\to\rho}\frac{V'(z)}{A'(z)}\\
&+ (2r+1-\ell) X(\rho)^{\ell}V(\rho)^{r}W(\rho)^{2r-\ell} \lim_{z\to\rho}\frac{W'(z)}{A'(z)}\Bigg).
\end{align*}
Recall that $\rho\sim \nicefrac1{2n}$ and $\hat{A}(\rho)\sim 1 - \nicefrac1{\sqrt{2n}}$, as $\nti$, and
hence $X(\rho)\sim\nicefrac13$, $V(\rho)\sim1$ and $W(\rho)\sim2$, as $\nti$. Moreover,
\[
\lim_{z\to\rho}\frac{W'(z)}{A'(z)} 
= \lim_{z\to\rho}\frac{\hat{A}'(z)-(n-\gamma)}{A'(z)}\frac1{(1-(\hat{A}(\rho)-(n-\gamma)\rho))^2} = \frac1{2(1-(\hat{A}(\rho)-(n-\gamma)\rho))^2},
\]
because $\lim_{z\to\rho}\frac{\hat{A}'(z)}{A'(z)} = \frac12$ and
$\lim_{z\to\rho}\frac{n-\gamma}{A'(z)} = 0$. Thus, $\lim_{n\to\infty} \lim_{z\to\rho}
\frac{W'(z)}{A'(z)} = 2$. Using similar arguments, we can prove that
\[
\lim_{\nti}\lim_{z\to\rho}\frac{V'(z)}{A'(z)} = 8 \quad \text{ and }\quad
\lim_{z\to\rho}\frac{X'(z)}{A'(z)}\sim \frac{8n}{9}, \text{ as }\nti.
\]
All these relations imply
\[
\lim_{z\to\rho}\frac{F'(z)}{A'(z)} \sim \frac{\kappa}{n^{\ell+1}}, \text{ as }\nti, 
\]
where $\kappa$ is a positive constant. 
\end{proof}

We are now ready to prove Proposition~\ref{prop:irreductibles}. The two previous lemmas allow us
to consider only irreducible trees in which the simplification process cuts at least one non-leaf
tree and fewer than $L(f)$ single leaves.
Let us denote by $\mathcal{I}_1$ the set of irreducible trees for which the simplified tree
$t^{\star}$ has size $L(f)$, and by $\mathcal{I}_2$ the set of irreducible trees $t$ of $f$ such
that $t^{\star}$ has size at least $L(f)+1$.

\begin{lem}
\label{lem:1}
We have the following asymptotic result:
$E^{*}(\mathcal{I}_1) = o\left(\nicefrac1{n^{L(f)}}\right)$.
\end{lem}

\begin{proof}
Let $t\in\mathcal{I}_1$. To obtain $t^{\star}$, we have cut subtrees of $t$ and we can assume that
we have cut at least one large subtree and at most $L(f)$ single leaves. Assume first that we have
cut only one non-leaf subtree rooted at a node $\nu$ during the algorithm. Then, either this tree
contains an essential variable on its first-level leaves, or it belongs to $\bigcup_{k,\ell\geq
0}\bigcup_{r=0}^{\ell-1} \mathcal{M}^{\Gamma_f}_{k,\ell,r}$, where $\Gamma_f$ is the set of essential variables of $f$. 
Otherwise, we could find an assignment of inessential
variables such that we can cut the father of $\nu$ without changing the function computed by the
tree $t$. This new assignment of inessential variables leads to a different simplification of the
tree $t$ that will cut at least the single leaves that were cut before, plus the larger subtree
and its father: We thus obtain a tree of size less than $L(f)$ that computes $f$, 
which is impossible.

Any tree of $\mathcal{I}_1$ is obtained by expanding a tree $s$ of $\mathcal{S}_{L(f)}$ as follows:
\begin{itemize}
\item choose an integer $q\geq 1$ ($q$ represents the number of large trees that were cut during
the process described beforehand: $q\geq 1$ holds because of the remark before Lemma~\ref{lem:1}), 
\item if $q=1$, plug a tree from $\left(\bigcup_{k,\ell\geq
0}\bigcup_{r=0}^{\ell-1} \mathcal{M}^{\Gamma_f}_{k,\ell,r}\right) \cup \mathcal{A}_{\Gamma_f}$ at a
node of $s$ and at most $L(f)$ inessential leaves at other nodes of $s$ (where $\mathcal{A}_{\Gamma_f}$ 
is the set of trees containing at least one first-level leaf labelled by an
essential variable of $f$, \emph{cf.}\ Proposition~\ref{prop:A_alpha}),
\item else plug $q\geq 2$ non-leaf subtrees and at most $L(f)$ inessential leaves at nodes of $s$.
\end{itemize}
We are interested in expansions of trees from $\mathcal{I}_1$. In fact, since we do not impose any
restrictions on the trees we plug at nodes of $s$, we can consider only expansions in the nodes of
$s$, since expansions in the plugged trees are then already counted. 

The generating function of trees obtained by successive expansions of trees from $\mathcal{I}_1$,
denoted by $I_1(z)$ thus satisfies:
\begin{align*}
I_1(z) \prec S_{L(f)}z^{L(f)} \sum_{k\geq 0} \Bigg(
&\begin{pmatrix}3L(f)+1+k\\L(f)+1,k,2L(f)\end{pmatrix} M(z)(2(n-\gamma)z)^{L(f)} \hat{B}(z)^k \\
&+ \sum_{q\geq 2} \begin{pmatrix}3L(f)+q+k\\L(f)+q, k,2L(f)\end{pmatrix} \hat{B}(z)^{q+k} (2(n-\gamma)z)^{L(f)}\Bigg),
\end{align*}
where $k$ counts the number of successive expansions done into the irreducible tree, and where
$M(z)$ is the generating function of $\left(\bigcup_{k,\ell\geq
0}\bigcup_{r=0}^{\ell-1} \mathcal{M}^{\Gamma_f}_{k,\ell,r}\right) \cup \mathcal{A}_{\gamma}$. The
multinomial coefficient represents the number of choices for the places where we plug trees in $s$
and where we then do the expansions (the orders of the ``pluggings'' and of the expansions do not
matter, but expansions are done after the ``pluggings'').

The coefficient $\begin{pmatrix}3L(f)+1+k\\L(f)+1,k,2L(f)\end{pmatrix}$ counts the different ways to plug one tree from
$\left(\bigcup_{k,\ell\geq
0}\bigcup_{r=0}^{\ell-1} \mathcal{M}^{\Gamma_f}_{k,\ell,r}\right) \cup \mathcal{A}_{\gamma}$, $L(f)$ inessential leaves and $k$ $\hat B$-expansions into $s\in\mathcal S_{L(f)}$. Let us start from $s$, the first plugging can be done at the left of each edge of $s$, or at the right of every rightmost edge, or at every leaf. Note that $s\in\mathcal S_{L(f)}$ has at most $L(f)$ leaves and $L(f)$ edges. The number of different places for the first plugging is thus at most $3L(f)$. The second plugging can then be made in $3L(f)+1$ different places, and so on, and so forth. The fact that we first do the $L(f)+1$ plugging and then the $k$ expansions gives the multinomial coefficient.
The second binomial coefficient is given by the same reasonning with $L(f)+q$ pluggings and $k$ expansions.

Thanks to Proposition~\ref{prop:A_alpha} and Section~\ref{part:useful}, we know that the limiting
ratio of $\left(\bigcup_{k,\ell\geq
0}\bigcup_{r=0}^{\ell-1} \mathcal{M}^{\Gamma_f}_{k,\ell,r}\right) \cup \mathcal{A}_{\Gamma_f}$ 
has order $\mathcal{O}\left(\frac1{\sqrt{n}}\right)$ and that $M(\rho)\leq \hat{B}(\rho)\sim\frac1{\sqrt{2n}}$. 
Thus, when $n$ tends to infinity,
\begin{align}
\mu_n(E^*(\mathcal{I}_1))
%\leq 
%        & S_{L(f)}\rho^{L(f)}\sum_{k\geq 0}
%            \Bigg(    
%            \begin{pmatrix}3L(f)+1+k\\L(f)+1,k,2L(f)\end{pmatrix} \left(\hat{B}(\rho)^{k+1} + \frac{k}{2} \hat{B}(\rho)^k\right)\notag\\
%        &\hspace{3cm}   +\sum_{q\geq 2} \begin{pmatrix}3L(f)+q+k\\L(f)+q,k,2L(f)\end{pmatrix} \frac{q+k}{2}\hat{B}(\rho)^{q+k-1}\Bigg)\notag\\
\leq    & S_{L(f)}\rho^{L(f)}\hat{B}(\rho)\sum_{k\geq 0}
            \Bigg(
            \begin{pmatrix}3L(f)+1+k\\L(f)+1,k,2L(f)\end{pmatrix} \left(\hat{B}(\rho)^{k} + \frac{k}{2} \hat{B}(\rho)^{k-1}\right)\notag\\
        &\hspace{3cm}   +\sum_{q\geq 2} \begin{pmatrix}3L(f)+q+k\\L(f)+q,k,2L(f)\end{pmatrix} \frac{q+k}{2}\hat{B}(\rho)^{q+k-2}\Bigg).\label{eq:somme}
\end{align}
The first factor $S_{L(f)}\rho^{L(f)}\hat{B}(\rho)$ behaves as $\frac1{n^{L(f)+\nicefrac12}}$ when $n$ tends to infinity. Let us prove that the second term of the sum behaves as $\mathcal{O}(1)$ when $n$ tends to infinity.
Let us first focus on
\begin{align*}
\sum_{k\geq 0}\begin{pmatrix}3L(f)+1+k\\L(f)+1,k,2L(f)\end{pmatrix} \hat{B}(\rho)^{k}
&= \frac{(3L(f)+1)!}{(L(f)+1)!(2L(f))!} \sum_{k\geq 0}\begin{pmatrix}3L(f)+1+k\\k\end{pmatrix} \hat{B}(\rho)^{k}\\
&= \frac{(2L(f)+2)!}{(L(f)+1)!(2L(f))!} \frac1{(1-\hat{B}(\rho))^{3L(f)+2}}
\end{align*}
in view of \eqref{eq:binome}.
Very similar calculations lead to
\[\sum_{k\geq 0}\begin{pmatrix}3L(f)+1+k\\L(f)+1,k\end{pmatrix}\frac{k}{2} \hat{B}(\rho)^{k-1}
= \frac{(3L(f)+1)!}{2(3L(f)+1)!(2L(f)!)}\frac{3L(f)+2}{(1-\hat{B})^{3L(f)+3}}.\]
Moreover, using~\eqref{eq:binome} again,
\begin{align*}
\sum_{k\geq 0}\sum_{q\geq 2}& \begin{pmatrix}3L(f)+q+k\\L(f)+q,k,2L(f)\end{pmatrix} \frac{q}{2}\hat{B}(\rho)^{q+k-2}\\
&= \sum_{q\geq 2}\frac{(3L(f)+q)!}{(L(f)+q)!(2L(f))!}\frac{q}{2}\hat{B}(\rho)^{q-2} \sum_{k\geq 0} \begin{pmatrix}3L(f)+q+k\\k\end{pmatrix}\hat{B}(\rho)^{k}\\
%&= \sum_{q\geq 2}\frac{(3L(f)+q)!}{(L(f)+q)!(2L(f))!}\frac{q}{2}\hat{B}(\rho)^{q-2}
%\frac{1}{(1-\hat{B}(\rho))^{3L(f)+q+1}}\\
%&= \frac{1}{2(1-\hat{B}(\rho))^{3L(f)+3}}\sum_{q\geq
%2}\begin{pmatrix}3L(f)+q\\L(f)+q\end{pmatrix}
%q\left(\frac{\hat{B}(\rho)}{(1-\hat{B}(\rho))}\right)^{q-2}\\
&= \frac{1}{2(1-\hat{B}(\rho))^{3L(f)+3}}\sum_{q\geq 0}\begin{pmatrix}3L(f)+q+2\\L(f)+q+2\end{pmatrix}(q+2)\left(\frac{\hat{B}(\rho)}{(1-\hat{B}(\rho))}\right)^{q}\\
&= \frac{1}{2(1-\hat{B}(\rho))^{3L(f)+3}}\sum_{q\geq 0}
\left[\prod_{j=2}^{L(f)+2}\left(\frac{2L(f)+q+j}{q+j}\right)\right]
\begin{pmatrix}2L(f)+q+1\\q+1\end{pmatrix}(q+2)\left(\frac{\hat{B}(\rho)}{1-\hat{B}(\rho)}\right)^{q}\\
&\leq \frac{(1+2L(f))^{L(f)+1}}{2(1-\hat{B}(\rho))^{3L(f)+3}}\sum_{q\geq 0}
\begin{pmatrix}2L(f)+q+1\\q+1\end{pmatrix}(q+2)\left(\frac{\hat{B}(\rho)}{1-\hat{B}(\rho)}\right)^{q}\\
&= \frac{(1+2L(f))^{L(f)+1}}{2(1-\hat{B}(\rho))^{3L(f)+3}} \left(\frac{2L(f)+1}{\left(1-\frac{\hat{B}(\rho)}{1-\hat{B}(\rho)}\right)^{2L(f)+2}} + \frac1{\left(1-\frac{\hat{B}(\rho)}{1-\hat{B}(\rho)}\right)^{2L(f)+1}}\right),
\end{align*}
since $\frac{\hat{B}(\rho)}{1-\hat{B}(\rho)}$ is smaller than $1$ for large enough $n$. Similar calculations can be done for the last term of the sum~\eqref{eq:somme}, and we eventually get
\[\mu_n(E^*(\mathcal{I}_1)) = \mathcal{O}\left(\frac{1}{n^{L(f)+\nicefrac12}}\right).\qedhere\]
\end{proof}

\begin{lem}
\label{lem:2}
We have the following asymptotic result:
$E^{*}(\mathcal{I}_2) = o\left(\nicefrac1{n^{L(f)}}\right).$
\end{lem}

\begin{proof}
The generating function $I_2(z)$ of trees obtained by successive expansions of a tree from
$\mathcal{I}_2$ satisfies
\[I_2(z)\leq \sum_{\ell\geq L(f)+1} S_{\ell} z^{\ell} \sum_{k\geq 0} \sum_{q\geq 0}\begin{pmatrix}2\ell+L(f)+q+k\\L(f)+q,k,2\ell\end{pmatrix}\hat{B}(z)^{k+q} (2(n-\gamma)z)^{L(f)}\]
where $S_{\ell}$ is the number of trees, where $\ell$ is the size of the seminal tree in $\mathcal I_2$, and where the multinomial coefficient counts the numbers of different ways to plug first $L(f)$ leaves and $q$ non-leaf trees and then $k$ expansions in this seminal tree.
Thus, calculations in the same vein as those done in the proof of Lemma~\ref{lem:1} yield
\begin{align*}
\mu_n(E^*(\mathcal{I}_2))
&\leq \sum_{\ell\geq L(f)+1} S_{\ell}\rho^{\ell}\sum_{k\geq 0}\sum_{q\geq 0}\begin{pmatrix}2\ell +L(f)+q+k\\L(f)+q,k,2\ell\end{pmatrix}\frac{k+q}{2} \frac1{\sqrt{n}^{k+q-1}}\\
&=\mathcal{O}\left(\frac1{n^{L(f)+1}}\right).\qedhere
\end{align*}
\end{proof}

Lemmas~\ref{lem:1} and~\ref{lem:2} directly induce Proposition~\ref{prop:irreductibles}, and we are now able to complete the proof of Theorem~\ref{thm:general}.

\begin{proof}[Proof of Theorem~\ref{thm:general}]
The probabilities of the set of tautologies and that of literals are treated in
Sections~\ref{part:tauto} and~\ref{part:simple_x}, respectively. We now have to prove that for all Boolean function $f$,
$\mathbb{P}_n(f) = \Theta\left(\nicefrac1{n^{L(f)}}\right)$. 
In view of Proposition~\ref{prop:several_exp}, we have 
$\mathbb{P}_n(f) = \Omega\left(\nicefrac1{n^{L(f)}}\right)$, 
and \eqref{eq:irr} together with Proposition~\ref{prop:irreductibles} gives
$\mathbb{P}_n(f) = \mathcal{O}\left(\nicefrac1{n^{L(f)}}\right)$
which completes the proof of Theorem~\ref{thm:general}.
\end{proof}

\section{Conclusion and further work}

This article presents the first attempt to discuss the size definition in quantitative logics:
while the formula size is commonly used in the literature,
we believe that in the case of non binary trees, 
the tree size considered in this article is at least as natural.

We proved that this change of size/complexity notion does not affect the first order behaviour 
of the distribution induced on the set of Boolean functions by the uniform distribution on and/or
trees of a given (large) size and labelled on $n$ variables.
However, we have also exhibited how the large typical and/or tree has a very different 
shape when changing the size notion. Trees having more leaves are more and more likely 
as the number $n$ of variables increases. 
This change of typical shape has forced us to develop original, 
more intricate proofs for this new model.

Note that our main result only gives a $\Theta$--asymptotic whereas for the formula size associative trees model (see~\cite{GGKM15})
a stronger equivalent result can be obtained. We strongly believe that the following holds: for all Boolean function, there exists a constant $\lambda_f>0$ such that, when $n\to +\infty$, 
\[\mathbb P_n (f) \sim \frac{\lambda_f}{n^{L(f)}},\]
but this requires a much deeper understanding of the structure of tautologies 
and certainly even more technical calculations.

Considering non binary tree is a way to take into account the associativity of the logical connectives $\land$ and $\lor$.
Note that there is no reason - except technicality - justifying that we consider plane trees instead of non plane trees while the connectives $\land$ and $\lor$ are commutative.
Non plane associative trees have been studied in the formula size model (see~\cite{GGKM15}),
and we believe that the same could be done for the tree size model, although once again,
the technical level of the computations would considerably increase since in non plane models, 
the generating functions are not known explicitly.

\bibliographystyle{plain} 
\bibliography{boolean}

\end{document}